\documentclass[11pt,a4paper]{article}
\usepackage[latin,english]{babel}
\usepackage[utf8]{inputenc}
\usepackage[T1]{fontenc}
\usepackage{eufrak}
\usepackage{geometry}
\usepackage{amsmath}
\usepackage{bbm}
\usepackage{amssymb}
\usepackage{amsbsy}
\usepackage{amsthm}
\usepackage{makeidx}
\usepackage{mathtools}
\usepackage{mathabx, mathrsfs, dsfont}
\usepackage{cases}
\usepackage{braket}
\usepackage[toc,page]{appendix}
\usepackage{dirtytalk}
\usepackage{pdfsync}
\usepackage{hyperref}
\usepackage{graphicx}
\usepackage{epstopdf}
\usepackage{todonotes}
\usepackage{comment}
\usepackage{graphicx}
\usepackage{fancyhdr}
\usepackage{math}
\usepackage{cite}
\usepackage{color}
\usepackage{subcaption}
\mathtoolsset{showonlyrefs}
\usepackage[shortlabels]{enumitem}

\newcommand{\balpha}{{\boldsymbol{\alpha}}}

\newcommand{\btheta}{\boldsymbol{\theta}}

\renewcommand{\T}{\mathbb{T}}
\renewcommand{\meanval}[1]{\bE\left[#1\right]}

\renewcommand{\mod}{{\,{\rm mod}\,}}

\newcommand{\Tr}{\text{Tr}}

\newtheorem{Lemma}{Lemma}[section]
\newtheorem{theorem}[Lemma]{Theorem}
\newtheorem{proposition}[Lemma]{Proposition}
\newtheorem{corollary}[Lemma]{Corollary}
\newtheorem{remark}[Lemma]{Remark}

\numberwithin{equation}{section}

\author{G. Mazzuca\footnote{Department of Mathematics, The Royal Institute of Technology, Stockholm, Sweden. \newline
\textit{email:} mazzuca@kth.se}, R. Memin\footnote{UMPA ENS de Lyon, 
Lyon, France. \newline
\textit{email:} ronan.memin@ens-lyon.fr}}
\title{Large Deviations for Ablowitz-Ladik lattice, and the Schur flow}

\begin{document}
\maketitle
\begin{abstract}
    We consider the Generalized Gibbs ensemble of the Ablowitz-Ladik lattice, and the Schur flow. We derive large deviations principles for the distribution of the empirical measures of the equilibrium measures for these ensembles. As a consequence, we deduce their almost sure convergence. Moreover, we are able to characterize their limit in terms of the equilibrium measure of the Circular, and the Jacobi beta ensemble respectively. 
\end{abstract}

\section{Introduction}

     The defocusing Ablowitz-Ladik (AL) lattice is the system of ODEs 	
	\begin{equation}
		\label{eq:AL}
		i	\dot{\alpha}_j =-(\alpha_{j+1}+\alpha_{j-1}-2\alpha_j)+|\alpha_j|^2(\alpha_{j-1}+\alpha_{j+1})\,,
	\end{equation}
	that describe the evolution of the complex functions $\alpha_j(t)$, $j\in\Z$ and $t\in\R$, here $\dot{\alpha}_j =\dfrac{d \alpha_j}{dt}$. We assume $N$-periodic boundary conditions $\alpha_{j+N}=\alpha_j$, for all $j\in \Z$. For simplicity, we consider the case $N$ even, and, when not mentioned, the limits as $N\to\infty$ is taken along $N$ even. This system was introduced by Ablowitz and Ladik \cite{Ablowitz1974,Ablowitz1975} as a spatial discretization of the defocusing Nonlinear Schr\"odinger Equation (NLS)
	
		\begin{equation}
		i \partial_t \psi(x,t) = -\frac{1}{2}\partial^2_{x} \psi(x,t) +  \lvert \psi(x,t) \rvert^2 \psi(x,t).
	\end{equation}
	The NLS is a well-known integrable model \cite{ZS}, and the Ablowitz-Ladik lattice is one of the several discretizations that preserve integrability \cite{Nijhoff1}.
	
	It is straightforward to verify that the two quantities
	
	\begin{equation}
        \label{eq:K0_K1}
	    K^{(0)}:=\prod_{j=1}^{N}\left(1-|\alpha_j|^2\right),\quad K^{(1)} := - \sum_{j=1}^{N}\alpha_j\overline{\alpha}_{j+1}, 
	\end{equation}
	are constants of motion for the AL lattice. Since $K^{(0)}$ is conserved along the flow, it implies that if $|\alpha_j(0)|< 1$ for all $j=1,\ldots,N$, then $|\alpha_j(t)|<1$ for all times. Thus, {we can consider $\D^N$ as our phase space, where $\D=\{ z\in \C\, |\, |z|<1\}$.}
	
	On this phase space we consider the symplectic form $ \omega$  \cite{Ercolani,GekhtmanNenciu}
\begin{equation}
		\label{eq:symplectic_form}
		\omega = i\sum_{j=1}^{N}\frac{1}{\rho_j^2}\di \alpha_j\wedge\di \wo \alpha_j\,,\quad \rho_j=\sqrt{1-|\alpha_j|^2}\,.
		\end{equation} 
	The corresponding Poisson bracket is defined for functions $f,g \in \cC^\infty(\D^N)$ as 
	\begin{equation}
		\label{eq:poisson_bracket}
		\begin{split}
		\{f,g\} & = i \sum_{j=1}^{N}\rho_j ^2\left(\frac{\partial f}{\partial \wo \alpha_j}\frac{\partial g}{\partial \alpha_j} - \frac{\partial f}{\partial \alpha_j}\frac{\partial g}{\partial \wo \alpha_j}\right)\,.		\end{split}
	\end{equation} 
	Using this Poisson bracket, it is possible to rewrite the equations of motion \eqref{eq:AL} of the AL lattice in Hamiltonian form as

	\begin{equation}
		\label{eq:hamiltonian}
		\dot{\alpha}_j =\{\al_j,H_{AL}\},\quad 	H_{AL}(\balpha) =  - 2\ln(K^{(0)}) + K^{(1)} + \overline{ K^{(1)}}\,,
	\end{equation} 
	here $\balpha = (\alpha_1,\ldots,\alpha_N)$.

	\paragraph{Conserved quantities.}
	
	As we already mentioned, the AL lattice is an integrable model: this was proved  by Ablowitz and Ladik \cite{Ablowitz1973, Ablowitz1974}. Specifically, they were able to obtain a Lax pair for the Ablowitz-Ladik lattice by discretizing the $2\times 2$ Zakharov-Shabat Lax pair  of the cubic  nonlinear Schr\"odinger equation.
	
	Nenciu and Simon in \cite{Nenciu2005,Simon2005} constructed a new Lax pair for the Ablowitz-Ladik lattice, exploiting the connection of this system to the orthogonal polynomials on the unit circle. This link is the analogue of the well-known link between the Toda lattice and orthogonal polynomials on
	  the real line  (see e.g. \cite{Deift}).  This connection was also generalized  to the non-commutative case \cite{Cafasso}. 
	   
	 Following \cite{Nenciu2005,Simon2005},
	 we construct the Lax matrix as follows.
    Consider the $2\times2$ unitary matrices
    
	\begin{equation}
	\label{matrix_Xi}
		\Xi(\alpha_j)=\Xi_j = \begin{pmatrix}
			\wo \alpha_j & \rho_j \\
			\rho_j & -\alpha_j
		\end{pmatrix}\,, \quad \rho_j = \sqrt{1-|\alpha_j|^2} ,\quad j=1,\dots, N\, ,
	\end{equation}
	and the $N\times N$ matrices \begin{equation}
	\label{eq:cM_cL}
		\cM= \begin{pmatrix}
			-\alpha_{N}&&&&& \rho_{N} \\
			& \Xi_2 \\
			&& \Xi_4 \\
			&&& \ddots \\
			&&&&\Xi_{N-2}\\
			\rho_{N} &&&&& \wo \alpha_{N}
		\end{pmatrix}\, ,\qquad 
		\cL = \begin{pmatrix}
			\Xi_{1} \\
			& \Xi_3 \\
			&& \ddots \\
			&&&\Xi_{N-1}
		\end{pmatrix} \,.
	\end{equation}
Now  let us define the Lax matrix 
	\begin{equation}
		\label{eq:Lax_matrix}
		\cE  = \cL \cM\,,
	\end{equation}
which has the following structure
	\[
	\cE = \begin{pmatrix}
	*&*&*&&&&&&&*\\
	*&*&*&&&&&&&*\\
	&*&*&*&*&&&&&\\
	&*&*&*&*&&&&&\\
	&&&&&\ddots&\ddots&&&\\
	&&&&&&*&*&*&*\\
	&&&&&&*&*&*&*\\
	*&&&&&&&*&*&*\\
	*&&&&&&&*&*&*\\
	\end{pmatrix}\,.
	\]
		The matrix $\cE$ is a periodic  CMV  matrix 
	 (after Cantero, Moral and Velazquez \cite{Cantero2005}).
It is straightforward to verify that the equations of motions \eqref{eq:AL}  are equivalent to  the following Lax equation for the matrix $\cE$:	
	\begin{equation}
		\label{eq:Lax_pair}
		\dot \cE = i\left[\cE, \cE^+ + (\cE^+)^\dagger\right]\,,
	\end{equation}	
	where $^\dagger$ stands for hermitian conjugate, and 
	\begin{equation}
		\cE^+_{j,k} = \begin{cases}
			\frac{1}{2} \cE_{j,j} \quad j = k \\
			\cE_{j,k} \quad k = j + 1 \, \mod \, N \, \mbox{or} \, k = j + 2 \, \mod \, N  \\
			0 \quad \mbox{otherwise}.
		\end{cases}
	\end{equation}

	{
	\begin{remark}
	\label{rem:unitary}
	    	We notice that since all the $\Xi_j$ are unitary, then also $\cE$ is unitary, this implies that all the eigenvalues $\lambda_j$ lie on the unit circle, and they can be written in terms of their argument, namely for all $j=1,\ldots,N$ there exists a $\theta_j\in \T := [-\pi,\pi) $ such that $$\lambda_j = e^{i\theta_j}\,.$$ 
	  
	    	In view of this identification, and in order to simplify the notations, for any function $f(z)\;:\; \partial \D \to\R$, we write $f(\theta)$ in place of $f(e^{i\theta})$ when it is convenient. Further, we will write indifferently $\int_{\T}f(\theta)\di\mu(\theta) $ or $\int_{\partial \D}f(z)\di\mu(z)$ for  any probability measure $\mu$ having support on the circle.
	\end{remark}
}

{
\begin{remark}
    We notice that $(\cE^+)^{\dagger}+(\cE^{\dagger})^+=\cE^{\dagger}$ and $[\cE,\cE^{\dagger}]=0$ since $\cE$  is unitary. Therefore, the Lax pair \eqref{eq:Lax_pair} can be rewritten
	in the equivalent form
		\begin{equation}
		\label{eq:Lax_pair1}
		\dot \cE = i\left[\cE, \cE^+ -(\cE^\dagger)^+\right]\,.
	\end{equation}
\end{remark}}
	The formulation \eqref{eq:Lax_pair} implies that the quantities 
	\begin{equation}
		\label{eq:constant_motion}
		K^{(\ell)}=\mbox{Tr}\left(\cE^\ell\right),\quad \ell=1,\dots, N-1,
	\end{equation}
	are constants of motion for the defocusing AL system \eqref{eq:AL}.
	
	{ As in \cite{spohn2021hydrodynamic,mazzuca2021generalized}, we introduce the Generalized Gibbs ensemble for the Ablowitz-Ladik lattice, namely the following probability measure on the phase space $\D^N$
	
	\begin{comment}
	\begin{equation}
	\label{GGE AL}
	   \di\mathbb{P}^{V,\beta}_{AL,N}(\alpha_1,\ldots,\alpha_N)= \frac{1}{Z_N^{AL}(V,\beta)}\prod_{j=1}^N(1-|\alpha_j|^2)^{\beta-1}\exp(-\Tr(V(\mathcal{E})))\di^2\balpha, \quad \alpha_j\in\D\; \forall\,j\,,
	\end{equation}
	\end{comment}
	{
	\begin{equation}
	\label{GGE AL}
	   \di\mathbb{P}^{V,\beta}_{AL,N}(\alpha_1,\ldots,\alpha_N)= \frac{1}{Z_N^{AL}(V,\beta)}\prod_{j=1}^N(1-|\alpha_j|^2)^{\beta-1}\mathbf{1}_{\{\alpha_j\in\D\}}\exp(-\Tr(V(\mathcal{E})))\di^2\balpha, \quad\,
	\end{equation}}
	where $V(e^{i\theta})\,:\, \T \to \R$ is a continuous function, {$\mathbf{1}_{A}$ is the indicator function of the set $A$}, and $Z_N^{AL}(V,\beta)$ is the partition function of the system 
	$$Z_N^{AL}(V,\beta) = \int_{\D^N}\prod_{j=1}^N(1-|\alpha_j|^2)^{\beta-1}\exp(-\Tr(V(\mathcal{E})))\di^2\balpha. $$
	Furthermore, 
	we consider the empirical measure $ \mu_N(\cE)$ of the eigenvalues $e^{i\theta_1},\ldots, e^{i\theta_N}$ of the matrix $\cE$ \eqref{eq:Lax_matrix}, namely

\begin{equation}
    \mu_N(\cE) = \frac{1}{N}\sum_{j=1}^N \delta_{e^{i\theta_j(\cE)}}\,,
\end{equation}
here $\delta_x$ is the delta function centred at $x${, furthermore, we notice} that we can just consider the arguments $\theta_1, \ldots,\theta_N$ of the eigenvalues since the matrix $\cE$ is unitary, see Remark \ref{rem:unitary}. 

Our main result is a large deviations principle (LDP) with good rate function for the sequence $(\mu_N(\cE))$ under the law $\mathbb{P}_{AL,N}^{V,\beta}$ \eqref{GGE AL}. }{Namely, denoting by $\cP(\T)$ the set of probability measures on the Torus $\T$ endowed with the topology of weak convergence, there exists a function $J_\beta^V:\cP(\T)\to [0,+\infty]$ such that:
\begin{enumerate}
    \item it is \textit{lower semicontinuous/good}, namely for any $a\geq 0$, $\{ \mu\in \cP(\T)\ |\ J_\beta^V(\mu) \leq a \}\subset \cP(\T)$ is compact, \\
    \item it satisfies a \textit{large deviations lower bound}, namely for all $O\subset \mathcal{P}(\T)$ open,
    \begin{equation}
    \label{ineq:large devations lower bound}
        -\inf_O J \leq \liminf_{N\text{ even}} \frac{1}{N} \ln \mathbb{P}(\mu_N(\cE) \in O),
    \end{equation}
    \item it satisfies a \textit{large deviations upper bound}, namely for all $F\subset \mathcal{P}(\T)$ closed,
    \begin{equation}
    \label{ineq:large devations upper bound}
        \limsup_{N\text{ even}} \frac{1}{N} \ln \mathbb{P}(\mu_N(\cE) \in F) \leq -\inf_F J.
    \end{equation}
\end{enumerate}
We refer to \cite{DZ} for a general introduction to large deviations. 
\begin{remark}
\label{rem:weak ldp is the same as strong ldp}
    We notice that, by compactness of $\cP(\T)$, it is sufficient to prove a \textbf{weak} large deviations principle, see \cite[Section 1.2]{DZ}, which is equivalent to a  full large deviations principle, except that the large deviation upper bound (point 3) holds only for \textbf{compact} subsets of $\cP(\T)$.
\end{remark}
From this large deviations principle we are able to deduce that $\mu_N(\cE)$ converges almost surely as $N$ goes to infinity.}

\begin{theorem}
\label{THM:MAIN_1}
Let $\beta >0$. For any continuous function $V\; :\; \T\; \to \R$  the following holds.
\begin{enumerate}[a.]
    \item the sequence $\mu_N(\cE)$ under the law $\mathbb{P}_{AL,N}^{V,\beta}$ satisfies a large deviations principle at speed $N$ with a good rate function $J_\beta^V$,
    \item $J_\beta^V$ achieves its minimum at a unique probability measure $\nu_\beta^V$,
    \item $\mu_N(\cE)$ converges almost surely and in $L^1(\T)$ towards $\nu_\beta^V$.
\end{enumerate}
\end{theorem}
Moreover, following \cite{mazzuca2021generalized,spohn2021hydrodynamic,GMToda}, we are able to characterize the measure $\nu_\beta^V$ in terms of the equilibrium measure of the Circular beta ensemble at high temperature \cite{Killip2004,Hardy2021}. More precisely, {consider the functional $\mu \mapsto f_\beta^V(\mu)$ given, for any $\mu\in \cP(\T)$ absolutely continuous with respect to Lebesgue measure and with density $\frac{\di \mu}{\di \theta}$, by}

\begin{equation}
    f_\beta^V(\mu) =  -\beta \int_{\T\times \T } \log\left( |e^{i\theta}- e^{i\varphi}| \right)\mu(\di \theta)\mu(\di \varphi)   + \beta\log(2) + \int_\T V(\theta) \mu(\di\theta) + \int_\T \log\left({\frac{\di \mu}{\di \theta}(\theta)}\right)\mu(\di \theta) + \log(2\pi)\,.
\end{equation}
It is shown in \cite{Hardy2021}, that the previous functional reaches its minimum for a unique absolutely continuous probability measure $\mu_\beta^V$. Moreover in \cite{mazzuca2021generalized} it is proved that this measure is almost surely differentiable with respect to $\beta$. Exploiting this result, and Theorem \ref{THM:MAIN_1} we are able to show that

\begin{theorem}
\label{THM:FINAL_RELATION}
For any continuous $V,f\, :\T\, \to \R$

\begin{equation}
    \int_\T f(\theta)\di \nu_\beta^V(\theta) = \partial_\beta\left( \beta\int_\T f(\theta)\di \mu_\beta^V(\theta)\right)\,.
\end{equation}
\end{theorem}
Thus, we obtain a unique characterization of the measure $\nu_\beta^V$.

%{This last theorem generalized the result of \cite{spohn2021hydrodynamic,mazzuca2021generalized}. More specifically, in these papers the authors considered the GGE \eqref{GGE AL} with polynomial potential, and they were able to prove Theorem \eqref{THM:FINAL_RELATION} for this particular class of potentials. Thus, we generalized their result for a wider class of potentials. }
{In \cite{spohn2021hydrodynamic,mazzuca2021generalized}, the authors considered the GGE \eqref{GGE AL} with polynomial potential, and they were able to prove Theorem \eqref{THM:FINAL_RELATION} for this particular class of potentials using a transfer operator technique. In this sense, we generalize their result, extending it to the class of continuous potentials.}

In the last part of the manuscript, we consider another integrable model related to the Ablowitz-Ladik lattice, namely the so-called Schur flow \cite{Golinskii}. Also for this system, the Lax matrix is $\cE$ \eqref{eq:Lax_matrix}. Following the same construction as in the Ablowitz-Ladik lattice case, we define a GGE for this model. We are able to show analogous results to Theorem \ref{THM:MAIN_1} and Theorem \ref{THM:FINAL_RELATION} for the Schur flow. The main difference is that in place of the Circular beta ensemble, we have the Jacobi one.

{ To give a wider overview of the relevant literature, we  mention that, H. Spohn in \cite{spohn2021hydrodynamic}, applying the theory of generalized hydrodynamics \cite{Doyon_notes}, argues that the correlation functions of the Ablowitz-Ladik lattice with respect to the GGE \eqref{GGE AL} show a ballistic behaviour. {As we already mentioned, }in \cite{mazzuca2021generalized} the authors rigorously proved Theorem \ref{THM:FINAL_RELATION} for polynomial potential $V(z)$. Moreover, they computed explicitly the density of states in the case $V(z) = \eta(z + \wo z)$, which corresponds to the classical Gibbs ensemble.

Lastly, it is worth to mention that this link between random matrix and integrable system was first noticed in \cite{Spohn1}. In this paper the author considered the GGE for the Toda lattice, and he was able to study this ensemble, comparing it with the Gaussian beta ensemble \cite{dued1}. We refer to \cite{Spohn2,Spohn3,Spohn4,mazzuca2021mean,GMToda,mazzuca2021generalized} for subsequent developments.

In particular, our work was inspired by the recent paper \cite{GMToda}. In this paper, the authors obtained a large deviations principle for the Toda lattice, and obtain an analogous result to Theorem \ref{THM:FINAL_RELATION}, where in place of the Circular beta ensemble, they had the Gaussian one. }

The structure of the paper is the following. In Section \ref{sec:weak_LPD}, we prove {the first point of Theorem \ref{THM:MAIN_1}.} In Section \ref{sec Circular}, we collect some known results related to the Circular beta ensemble in the high-temperature regime. Moreover, we  reformulate the already known large deviations principle for this ensemble in terms of the AL lattice. In Section \ref{sec main thm}, we  conclude the proof of Theorem \ref{THM:MAIN_1}, and we prove Theorem \ref{THM:FINAL_RELATION}. Section \ref{sec: Schur}, is dedicated to the Schur flow, where we prove the analogue of Theorem \ref{THM:MAIN_1} and Theorem \ref{THM:FINAL_RELATION} for this integrable model. Finally, we defer to the appendix the most technical results of our manuscript.

\section[Weak LPD for AL]{Existence of a Large deviations principle for the empirical measure of the Ablowitz-Ladik lattice}
\label{sec:weak_LPD}
{
The aim of this section is to prove the first point of Theorem \ref{THM:MAIN_1}, namely to show that, for $\N\ni N \geq 2$ and even, the sequence of empirical measures $\mu_N(\cE) = \frac{1}{N}\sum_{j=1}^N \delta_{e^{i\theta_j(\cE)}} $ satisfies a large deviations principle. The strategy of proof is the following. First, we show that  if $\mathcal{E}$ is distributed according to $\mathbb{P}^\beta_{AL,N}:=\mathbb{P}^{0,\beta}_{AL,N}$ defined in \eqref{GGE AL}, then the sequence of random probability measures $(\mu_N(\cE))_{N \text{ even}}$ satisfies a large deviations principle in $\mathcal{P}(\T)$, the space of probability measures on $\T$, endowed with the topology of weak convergence. Since according to this toppology $\mathcal{P}(\T)$ is compact, it suffices to show that the sequence $(\mu_N(\cE))_{N \text{ even}}$ satisfies a \textbf{weak} large deviations principle, see Remark \ref{rem:weak ldp is the same as strong ldp}. Then, applying Varadhan's Lemma \cite[Theorem 1.2.1]{DupuisEllis}, we obtain the {existence of a large deviations principle for arbitrary continuous $V$, \textit{i.e.} the first point of Theorem \ref{THM:MAIN_1}}.

We also notice that when $V=0$ in \eqref{GGE AL} the $\alpha_i$'s are independent and identically distributed (\textit{i.i.d}) with distribution $\Theta_{2\beta+1}$, where $\Theta_\nu$ is defined for $\nu>1$ as the random variable such that for $f:\C\to \R$ bounded and measurable
\begin{equation}
\label{def theta}
    \E[f(X)]=\frac{\nu-1}{2\pi}\int_{\overline{\D}} f(z)(1-|z|^2)^\frac{\nu-3}{2} \di^2z.
\end{equation}}

\begin{remark}
\label{rem: interpretation}
    We recall that for integer $\nu\geq 2$,  such measure has the following geometrical interpretation:   if $\bu = (u_1, \ldots, u_{\nu+1})$ is chosen  at random according to the surface measure  on  the unit sphere $S^\nu$  in $\mathbb{R}^{\nu+1}$,   then $u_1+i u_2$ is $\Theta_\nu$ distributed \cite{Kilnen07}.
\end{remark}

%{We will see that in order to} {prove}
To show that the sequence $(\mu_N(\cE))_{N \text{ even}}$   satisfies a weak large deviations principle according to the law $\mathbb{P}_{AL,N}^\beta$, we only need the $\alpha_i$'s to be \textit{i.i.d} according to some law $\sigma$ with $\rm{supp}(\sigma)\subseteq \overline{\D}$. Thus, we just assume the latter hypothesis, and we prove the result in more generality.

\subsection{Large Deviations Principle for periodic CMV matrix}

Let $d$ be the distance on $\mathcal{P}(\T)$ defined by
\begin{equation}
\label{def_distance}
    d(\mu,\nu)=\sup_{\|f\|_{\text{Lip}}\leq 1, \|f\|_{\text{BV}\leq 1}} \left\{\left| \int fd\mu - \int fd\nu \right|\right\}.
\end{equation}
Where the Lipschitz and the bounded variation norms are defined on the space of functions $f:\T\to \R$ as
\begin{align}
    &\|f\|_{\textrm{Lip}}=\sup_{{\theta_1,\theta_2\in \T} \atop \theta_1\neq \theta_2 }\frac{|f\left(e^{i\theta_1}\right)-f\left(e^{i\theta_2}\right)|}{|e^{i\theta_1}-e^{i\theta_2}|}\,, \\
    &\|f\|_{\textrm{BV}}= \sup_{n\geq 1, 0=\theta_1 <\theta_2 <\ldots <\theta_n=2\pi}\sum_{k=1}^{n-1} \left | f\left(e^{i\theta_{k+1}}\right)-f\left(e^{i\theta_k}\right)\right | \,.
\end{align}

 The distance $d$ is compatible with the weak convergence of probability measures \cite{GMToda}. {We recall that for a $N\times N$ matrix $A$, its empirical measure of eigenvalues is defined by 
 $$\mu(A)=\frac{1}{N}\sum_{j=1}^N\delta_{\lambda_j(A)},$$
 where $\lambda_j(A)$, $j=1,\ldots,N$, are the eigenvalues of $A$.
 } The following Lemma, whose proof can be found in Appendix \ref{app:A}, gives an upper bound on the distance of the empirical measures of two unitary matrices.

\begin{Lemma}
\label{LEM:DISTANCE_INEQ}
For any $A$, $B$ unitary matrices of size $N\times N$,
\begin{itemize}
    \item For $f$ with bounded variation, 
    $$\left| \int fd\mu(A) - \int fd\mu(B) \right| \leq \|f\|_{\textrm{BV}}\frac{rank(A-B)}{N},$$
    \item For $f$ Lipschitz,
    $$\left| \int fd\mu(A) - \int fd\mu(B) \right| \leq \|f\|_{\textrm{Lip}}\frac{1}{N}\sum_{i,j=1}^N|(A-B)_{i,j}|.$$
\end{itemize}
As a consequence,
\begin{equation}
    \label{distance}
    d(\mu(A),\mu(B))\leq \min\left\{\frac{rank(A-B)}{N},\frac{1}{N}\sum_{i,j=1}^N|(A-B)_{i,j}| \right\}.
\end{equation}
\end{Lemma}

We are now in position to prove that the sequence $\mu_N(\cE)$ with $(\alpha_i)_{i\geq 1}$ $\textit{i.i.d}$ with law $\sigma$, such that $\text{supp}(\sigma)\subseteq \overline{\D}$, satisfies a large deviations principle. The proof of the following Lemma follows the same line as the corresponding one in \cite{GMToda}. 
\begin{Lemma}
\label{Lemma weak ldp}
Let $(\alpha_i)_{i\geq 1}$ be an $\textit{i.i.d}$ sequence of law $\sigma$ with $\text{supp}(\sigma)\subseteq \overline{\D}$, and $\mathcal{E}$ be the associated matrix defined in \eqref{eq:Lax_matrix}. Then the sequence of empirical measures $(\mu_N(\cE))_{N \text{ even}}$ satisfies a large deviations principle in $\mathcal{P}(\T)$ endowed with the topology of weak convergence.
\end{Lemma}

\begin{proof}
{We use a subadditivity argument to show that for any fixed $\mu \in \mathcal{P}(\T)$ the following holds
\begin{equation}
\label{eq weak ldp}
   \lim_{\delta\to 0}\liminf_{N\text{ even}}\frac{1}{N}\ln\mathbb{P}\bigg(\mu_N(\cE)\in B_{\mu}(\delta)\bigg) = \lim_{\delta\to 0}\limsup_{N\text{ even}}\frac{1}{N}\ln\mathbb{P}\bigg(\mu_N(\cE)\in B_{\mu}(\delta)\bigg),
\end{equation}
where $ B_{\mu}(\delta) := \left \{ \nu\in\cP(\T)\; \vert \; d(\mu,\nu)<\delta \right\} $.} Then, applying \cite[Theorem 4.1.11]{De-Ze}, along with the fact that in our setting a weak LDP is equivalent to a full LDP, due to the compactness of $\cP(\T)$, {see remark \ref{rem:weak ldp is the same as strong ldp},} we conclude. \\
The first step to prove the result is to approximate the matrix $\cE$ (whose law we denote by $\cE^{(N)}$) by a diagonal block matrix of independent blocks. To this end, 
fix $q\in \N$ even such that $q\leq N$, write the euclidean division of $N$ by $q$, $N=kq+r$ with $0\leq r<q$.
We consider $\cM$ given by \eqref{eq:cM_cL},
$$\cM= \begin{pmatrix}
			-\alpha_{N}&&&&& \rho_{N} \\
			& \Xi_2 \\
			&& \Xi_4 \\
			&&& \ddots \\
			&&&&\Xi_{N-2}\\
			\rho_{N} &&&&& \wo \alpha_{N}
		\end{pmatrix},$$
and approximate it the following way.

Let $\widetilde{\cM} = \text{diag}(\cM_1,\ldots, \cM_k,R)$, where $\cM_i$ is the block diagonal matrix given by
 $$\cM_i= \begin{pmatrix}
 			-\widetilde\alpha_{(i-1)q} &&&&& \widetilde\rho_{(i-1)q} \\
 			& \Xi_{(i-1)q+2} \\
 			&& \Xi_{(i-1)q+4} \\
 			&&& \ddots \\
 			&&&&\Xi_{iq-2}\\
 			\widetilde\rho_{(i-1)q} &&&&&  \overline{\widetilde\alpha}_{(i-1)q}
 		\end{pmatrix},$$
 where $(\widetilde{\alpha}_{(i-1)q})_{1\leq i\leq k}$ are \textit{i.i.d} of law $\sigma$, independent of the $\alpha_i$'s, $\widetilde\rho_i= \sqrt{1-|\widetilde\alpha_i|^2}$,
 and the remaining block (of size $r\times r$) $R$ is defined similarly:
 $$R= \begin{pmatrix}
 			-\widetilde{\alpha}_{kq} &&&&& \widetilde{\rho}_{k} \\
 			& \Xi_{kq+2} \\
 			&& \Xi_{kq+4} \\
 			&&& \ddots \\
 			&&&&\Xi_{N-2}\\
 			\widetilde{\rho}_{kq+1} &&&&&  \overline{\widetilde{\alpha}_{kq}}
 		\end{pmatrix}.$$
 Following the same decomposition of $N=kq+r$ we write $\cL=\text{diag}(\cL_1,\ldots,\cL_k,\cL_{k+1})$, with $\cL_i$ of size $q$ for $1\leq i \leq k$ and $\cL_{k+1}$ of size $r$.\\
 Notice that by construction, we have
 \begin{equation}
     \label{ineq_rank_tilde}
      \text{rank}(\cM - \widetilde{\cM}) \leq 2(k+1).
 \end{equation}
 Now, defining $\widetilde{\cE}=\cL \widetilde{\cM}$, $\widetilde{\cE}$ is a block diagonal matrix $\text{diag}(\cE_1,\ldots,\cE_k,\cE_{k+1})$.Then, the blocks $\cE_i$, $1\leq i \leq k+1$ are independent, each $\cE_i$, $i=i,\ldots,k$, has law $\cE^{(q)}$, and $\cE_{k+1}$ has law $\cE^{(r)}$. \\
 Furthermore, using that $\text{rank}(AB)\leq\min\left\{ \text{rank}(A);\text{rank}(B)\right\}$ for $A,B$ two square matrices, and \eqref{ineq_rank_tilde} we get
 $$\text{rank}(\cE - \tilde{\cE})=\text{rank}({\cL(\cM-\widetilde{\cM})})\leq 4(k+1).$$ 
 By the first point of Lemma \ref{LEM:DISTANCE_INEQ} we deduce
 $$ d(\mu_N(\cE), \mu_N(\widetilde{\cE}) ) \leq \frac{4(k+1)}{N}\leq \frac{8}{q}.$$
Moreover, we can rewrite $\mu_N(\wt \cE)$ as 
 
 \begin{equation}
     \mu_N(\wt \cE) = \frac{q}{N}\sum_{\ell = 1}^k \mu_q(\cE_\ell) + \frac{r}{N}\mu_r(\cE_{k+1})\,.
 \end{equation}
 Using the independence of the blocks of $\widetilde{\cE}$, we deduce that
 \begin{align*}
  \mathbb{P}\bigg(\mu_q(\cE_1)\in B_{\mu}(\delta)\bigg)^k\mathbb{P}\bigg(\mu_r(\cE_{k+1})\in B_{\mu}(\delta)\bigg) &=\mathbb{P}\bigg(\mu_q(\cE_1), \ldots, \mu_q(\cE_k), \mu_r(\cE_{k+1}) \in B_{\mu}(\delta)\bigg) \\
                        &\leq \mathbb{P}\bigg( {\frac{q}{N}}\sum_{l=1}^k\mu_{q}(\cE_l) + {\frac{r}{N}}\mu_r(\cE_{k+1}) \in B_\mu(\delta) \bigg) \\
                       &= \mathbb{P}\bigg( \mu_N(\tilde{\cE})\in B_\mu(\delta) \bigg) \\
                       &\leq \mathbb{P}\bigg( \mu_N(\cE) \in B_\mu\left(\delta + \frac{8}{q}\right) \bigg),
 \end{align*}
Where we used the convexity of balls in the first inequality.\\
This implies that, setting 
\begin{equation}
    u_N(\delta) = - \ln\left(\mathbb{P}(\mu_N \in B_\mu(\delta)) \right)\,,
\end{equation}
we have
\begin{equation}
    u_N\left(\delta +\frac{8}{q}\right) \leq k u_q(\delta) + u_r(\delta).
\end{equation}
We now conclude as in \cite[Lemma 2.3]{GMToda}.
Let $\delta>0$ and choose $q$ in such a way that $\frac{8}{q}\leq \delta$, so we deduce that 
\begin{equation}
    \frac{u_N(2\delta)}{N} \leq \frac{u_N\left(\delta+\frac{8}{q} \right)}{N} \leq \frac{u_q(\delta)}{q} + \frac{u_r(\delta)}{N}\,,
\end{equation}
since $\frac{u_r(\delta)}{N}\to 0$ as $N\to\infty$, we deduce that

\begin{equation}
    \limsup_{N\to\infty} \frac{u_N(2\delta)}{N} \leq \frac{u_q(\delta)}{q}\,.
\end{equation}
The previous inequality holds true for all $q$ big enough, so we conclude that
\begin{equation}
        \limsup_{N\to\infty} \frac{u_N(2\delta)}{N} \leq \liminf_{N\to \infty} \frac{u_N(\delta)}{N}\,.
\end{equation}
From this last inequality we deduce that
\begin{equation}
        \lim_{\delta\to0}\limsup_{N\to\infty} \frac{u_N(\delta)}{N} \leq \lim_{\delta\to0} \liminf_{N\to \infty} \frac{u_N(\delta)}{N}\,,
\end{equation}
thus we obtain \eqref{eq weak ldp}, and the conclusion follows applying \cite[Theorem 4.1.11]{De-Ze}.
\end{proof}

Since $(\mathcal{P}(\T),d)$ is compact, Lemma \ref{Lemma weak ldp} automatically implies the existence of a \textbf{strong} large deviations principle. Furthermore, the corresponding rate function $J$ ,which depends on the distribution $\sigma$ of the entries of $\cL$ and $\cM$, can be seen to be convex. We collect these results in the following proposition. 
\begin{proposition}
\label{prop large deviations}
In the same hypothesis and notations as in Lemma \ref{Lemma weak ldp}, the sequence of empirical measures $(\mu_N( \cE))_{N \text{ even}}$ satisfies a large deviations principle with good, convex rate function $J:\mathcal{P}(\T)\to [0,+\infty]$, \textit{i.e.}
\begin{itemize}
    \item The function $J$ is convex and its level sets $J^{-1}([0,a])$, $a\geq 0$, are compact,
    \item For all $O\subset \mathcal{P}(\T)$ open,
    $$ -\inf_O J \leq \liminf_{N\text{ even}} \frac{1}{N} \ln \mathbb{P}(\mu_N(\cE) \in O), $$
    \item For all $F\subset \mathcal{P}(\T)$ closed,
    $$  \limsup_{N\text{ even}} \frac{1}{N} \ln \mathbb{P}(\mu_N(\cE) \in F) \leq -\inf_F J.$$
\end{itemize}
\end{proposition}

\begin{proof}
We already established all the claims except the fact that the function $J$ is convex and that the level sets $J^{-1}([0,a])$ are compact. The latter comes from the fact that these sets are closed, see \cite[Theorem 4.1.11]{De-Ze}. To prove the convexity of $J$, we follow the same argument as \cite[Theorem 2.4]{GMToda}.

Let $\mu_1,\mu_2\in\cP(\T)$. Since $\mu_{2N}(\cE)$ can be approximated by the sum of two independent $\mu_N(\cE)$'s up to a mistake smaller than $\frac{4}{N}$ {by the first point of Lemma \ref{LEM:DISTANCE_INEQ}}, for $\delta > 0$ the following holds

\begin{equation}
    \mathbb{P}\left( \mu_{N}(\cE) \in B_{\mu_1}(\delta)\right)\mathbb{P}\left( \mu_{N}(\cE) \in B_{\mu_2}(\delta)\right) \leq \mathbb{P}\left( \mu_{2N}(\cE) \in B_{\frac{\mu_1+\mu_2}{2}}\left(\delta +\frac{4}{N}\right)\right)\,,
\end{equation}
taking minus the logarithm of both sides, dividing by $2N$, taking the limit for $N$ going to infinity and then for $\delta$ to zero, we deduce that:

\begin{equation}
    J\left( \frac{\mu_1+\mu_2}{2}\right) \leq \frac{1}{2}\left( J(\mu_1) + J(\mu_2)\right)\,,
\end{equation}
{ which, together with the lower semi-continuity of $J$, implies the convexity of $J$, see \cite[Lemma 4.1.21]{DZ} .} 
\end{proof}

\subsection[LDP for AL]{Large deviations principle for the Ablowitz-Ladik lattice}

Taking $\sigma=\Theta_{2\beta+1}$ given by equation \eqref{def theta}, Proposition \ref{prop large deviations} applies to $(\mu_N(\cE))_{N\text{ even}}$, where $\cE$ follows $\mathbb{P}^\beta_{AL,N}$ defined in \eqref{GGE AL}. Thus, $(\mu_N(\cE))_{N\text{ even}}$ with law $\mathbb{P}_{AL,N}^\beta$ satisfies a large deviations principle, with a good convex rate function, that we  denote by $J_\beta$.

We can now state the existence of a large deviations principle for $(\mu_N(\cE))_{N\text{ even}}$ under $\mathbb{P}^{V,\beta}_{AL,N}$ for $V$ continuous. 
\begin{corollary}
\label{cor: LDP AL}
Let $\beta >0$, and  $V:\T\to \R$ be continuous. Under $\mathbb{P}^{V,\beta}_{AL,N}$ the sequence $(\mu_N(\cE))_{N\text{ even}}$ fulfils a large deviations principle with good, convex rate function {$J^V_\beta(\mu) = g_\beta^V(\mu) - \inf_{\nu\in\cP(\T)}g_\beta^V(\nu)$, where $g_\beta^V(\mu)$ is  given for $\mu\in \mathcal{P}(\T)$ by}
\begin{equation}
\label{eq: rate function AL}
    {g^V_\beta(\mu) = J_\beta(\mu) + \int_\T V\di\mu \,.}
\end{equation}
\end{corollary}

\begin{proof}
Let us write
$$ \di \mathbb{P}^{V,\beta}_{AL,N}=\frac{Z_N^{AL}(0,\beta)}{Z_N^{AL}(V,\beta)}e^{-N\int_\T V\di\mu_N}\di\mathbb{P}^{\beta}_{AL,N} = \frac{1}{\cZ ^{AL,V}_N}e^{-N\int_\T V\di\mu_N}\di\mathbb{P}^{\beta}_{AL,N}.$$
The function $\mu \mapsto \int_\T V \di\mu$ being bounded continuous, by the large deviations principle under $\mathbb{P}^{\beta}_{AL,N}$ and Varadhan's Lemma, \cite[Theorem 1.2.1]{DupuisEllis}, we see that for any bounded continuous $f:\cP(\T)\to \R$ we have

$$\lim_N \frac{1}{N}\ln \int_{\D^N} e^{Nf(\mu_N})\di\mathbb{P}^{V,\beta}_{AL,N} = \sup_{\mu\in \cP(\T)}\left\{ f(\mu)-\left( J_\beta(\mu)+\int V\di\mu-{\inf_{\nu\in\cP(\T)}\left\{ J_\beta(\nu)+\int V\di\nu\right\}} \right) \right\},$$

which ensures by \cite[Theorem 1.2.3]{DupuisEllis} that $(\mu_N)$ satisfies a large deviations principle under $\mathbb{P}^{V,\beta}_{AL,N}$ with the announced rate function. Since the function $J^V_\beta$ is an affine perturbation of $J_\beta$, which is convex, $J^V_\beta$ is also convex.
\end{proof}

The first point of Theorem \ref{THM:MAIN_1} is proven.

\section{Circular $\beta$ ensemble at high temperature}
\label{sec Circular}

{
In this section, we consider the Circular $\beta$ ensemble, and we collect some known results that we exploit in our treatment. The aim of this section is to prove an alternative formulation of the large deviations principle for the Circular beta ensemble in the high-temperature regime, see Theorem \ref{thm:alternative LDP} below. Our formulation allows us to relate the large deviations principle of the Coulomb gas with the one of Ablowitz-Ladik, proved in the previous section.}

\subsection{Large deviations principle for Circular $\beta$ ensemble}

Coulomb gas on the torus $\T=[-\pi,\pi)$ at temperature $\wt \beta^{-1}$ are described by

\begin{equation}
\label{eq:circular_distr}
    \di \mathbb{P}^{V,\wt \beta}_{C,N} = \frac{1}{Z_{N}^C(V,\wt\beta)} \prod_{j\leq \ell\leq N} |e^{i\theta_j} - e^{i\theta_\ell}|^{\wt \beta } e^{\sum_{j=1}^N V(\theta_j)}\di \btheta\,,
\end{equation}
here $V\,:\, \T \; \to \R$ is a continuous potential, and $\btheta = (\theta_1, \ldots, \theta_N)$. When $V=0$, Killip and Nenciu showed that  $\di \mathbb{P}^{0,\wt \beta}_{C,N}$ is the law of the eigenvalues of a CMV matrix \cite{Killip2004}, see Theorem \ref{thm:Killipenciu}. In this manuscript, we are interested in the so-called \textit{high-temperature regime} of this ensemble, namely the limit of number of particles $N$ going to infinity with the constraint that $\wt \beta N \to 2\beta > 0$. This regime was considered by Hardy and Lambert in \cite{Hardy2021}, who proved the following large deviations principle for the measure $ \mu_N = \frac{1}{N}\sum_{j=1}^N \delta_{e^{i\theta_j}}$, where the $\theta_j$ are distributed according to \eqref{eq:circular_distr}.

\begin{theorem}
\label{thm_LDPCoulomb}
Let $\wt \beta = \frac{2\beta}{N}$, $\beta > 0$ and assume $V \; :\; \T \to \R$ to be continuous. Define for any $\mu\in\cP(\T)$ absolutely continuous with respect to the Lebesgue measure the functional
\begin{equation}
    f_\beta^V(\mu) =  -\beta \int_{\T\times \T } \log\left( |e^{i\theta}- e^{i\varphi}| \right)\mu(\di \theta)\mu(\di \varphi)   + \beta\log(2) + \int_\T V(\theta) \mu(\di\theta) + \int_\T \log\left({\frac{\di\mu}{\di\theta}(\theta)}\right)\mu(\di \theta) + \log(2\pi)\,,
\end{equation}
then
\begin{itemize}
    \item[i.] the functional $f_\beta^V(\mu)$ is strictly convex and achieves its minimal value at the unique probability measure  $\mu_\beta^V$ absolutely continuous with respect to the Lebesgue measure;
    \item[ii.] the sequence $( \mu_N)$ satisfies a large deviations principle in $\cP(\T)$ equipped with the weak topology at speed $\beta N$ with rate function {defined for absolutely continuous $\mu\in\cP(\T)$ with respect to Lebesgue measure by $I^V_\beta(\mu) = f_\beta^V(\mu) - f_\beta^V(\mu_\beta^V)$, and $I^V_\beta(\mu)=+\infty$ otherwise.} In particular
    \begin{equation}
         \mu_N \xrightarrow[N\to\infty]{\text{a.s.}} \mu_\beta^V\,.
    \end{equation}
\end{itemize}
\end{theorem}

Exploiting this result, in \cite{mazzuca2021generalized} the authors deduced several useful properties of the minimizer $\mu_\beta^V$, specifically they proved the following.

\begin{Lemma}[cf. \cite{mazzuca2021generalized} Lemma 3.5  ]
\label{lem: prop circ}
Let $\beta >0$, consider a continuous potential $V\,: \, \T \to \R$,  then the following holds
	    \begin{itemize}
	        \item[i.] The map $\beta \to \inf \left( f^{V}_\beta(\mu)\right)$ is Lipschitz;

	        \item[ii.]Let $D$ be the distance on $\cP(\T)$ given by
	        \begin{equation}
	        \label{eq:distance}
	        \begin{split}
	            D(\mu, \mu')  & = \left( - \int_{\T\times\T}\ln\left\vert \sin \left( \frac{\theta - \phi}{2}\right)\right \vert (\mu - \mu')(\di\theta)(\mu - \mu')(\di \phi)\right)^{1/2}\\ 
	            & = \sqrt{\sum_{k\geq 1} \frac{1}{k}\left\vert \wh \mu_k - \wh \mu'_k \right \vert^2}\,,
	        \end{split}
	        \end{equation}
	        where $\wh \mu_k = \int_\T e^{ik\theta} \mu(\di \theta)$. 
	        Then for any $\varepsilon>0$ there exists a finite constant $C_\varepsilon$ such that for all $\beta,\beta'>\varepsilon$
	        
	        \begin{equation}
	        \label{Distance}
	            D(\mu_\beta^V, \mu_{\beta'}^V) \leq C_\varepsilon \left\vert \beta - \beta' \right \vert\,.
	        \end{equation}
	        
	    \end{itemize}
	
\end{Lemma}
	\begin{remark}
	    \label{rem:lip_moments}

	We observe that if $f\in L^2(\T)$ with derivative in $L^2(\T)$, we can set $|| f||_{\frac{1}{2}} = \sqrt{\sum_{k\geq 1} k |\wh f_k|^2}$. So, for any measure $\nu$  with zero mass we obtain the following 

	\begin{equation}
	    \int_{\T} f(\theta) \nu(\di\theta) = \sum_{k\ne 0} \wh f_k \overline{\wh \nu_k} = \sum_{k\ne 0} k\wh f_k \frac{\overline{\wh \nu_k}}{k}\,.
	\end{equation}
	Then, by Cauchy-Schwartz inequality, we deduce the following inequality 
	
	\begin{equation}
	\label{eq:CS}
	    \left \vert \int_{\T} f(\theta) \nu(\di\theta) \right \vert^2 \leq \left \vert \sum_{k\ne 0} k\wh f_k \sum_{k\ne 0} \frac{\overline{\wh \nu_k}}{k}\right \vert  \leq 4  ||f||_{\frac{1}{2}}^2 D(\nu, 0)^2.
	\end{equation}
	Combining \eqref{Distance} and \eqref{eq:CS}, we deduce that  for any function $f$ with finite $||f||_{\frac{1}{2}}$ norm,  the map $\beta \to \int_{\T}f\di \mu_\beta^V(\theta)$ is Lipschitz for $\beta > 0$.
	\end{remark}

\subsection{Relation with the large deviations principle of the Ablowitz-Ladik lattice}

{
In the case $V=0$, for any $\wt \beta >0$, Killip and Nenciu in \cite{Killip2004} showed that the law $\mathbb{P}_{C,N}^{0,\wt \beta}$ \eqref{eq:circular_distr} coincides with the distribution of the eigenvalues of a certain CMV matrix. Specifically they proved the following:

\begin{theorem}[cf. \cite{Killip2004} Theorem 1.2]
	\label{thm:Killipenciu}
	Consider the block diagonal $N\times N$ matrices
	\begin{equation}
		\label{eq:LME}
		L = \mbox{diag}\left(\Xi_1,\Xi_3,\Xi_{5} \ldots,\right) \quad \mbox{ and } \quad M = \mbox{diag}\left(\Xi_{0},\Xi_2,\Xi_4, \ldots\right)\,,
	\end{equation}
	where the block $\Xi_j$, $j=1,\dots, N-1$, takes the form 
	\begin{equation}
		\label{eq:xi_def}
		\Xi_j = \begin{pmatrix}
			\wo \alpha_j & \rho_j \\
			\rho_j & -\alpha_j
		\end{pmatrix}\, ,\;\;\rho_j = \sqrt{1-|\alpha_j|^2}, 
	\end{equation}	
	while $\Xi_{0} = (1)$ and $\Xi_{N} = (\wo \alpha_{N})$ are $1\times 1$ matrices. 
	Define the  $N \times N$ sparse matrix 
	\begin{equation}
	\label{E}
	E = LM,
	\end{equation}
	and suppose that  the entries $\alpha_j $ are independent complex random variables with $\alpha_j\sim \Theta_{\wt \beta(N-j) +1 }$ 	 for $1\leq j\leq N-1$ and  $\alpha_{N}$ is  uniformly distributed on the unit circle.
	Then the eigenvalues of $E$ are distributed according to the Circular  Ensemble  \eqref{eq:circular_distr} at temperature $\tilde{\beta}^{-1}$. 
	
\end{theorem}

To simplify the notation, we will denote by $\mathbb{P}_{C,N}^{\wt \beta}$ the law $\mathbb{P}_{C,N}^{0,\wt \beta}$. We give an alternative {formulation of the} large deviations principle for the empirical measure under the law $\mathbb{P}_{C,N}^{ \frac{2\beta}{N}}$ based on the Killip-Nenciu matrix representation. This alternative {formulation} allows us to relate the rate function of the Coulomb gas $I_\beta$ in terms of the rate function $J_\beta$ of the Ablowitz-Ladik lattice. Finally, applying Varadhan's Lemma \cite[Theorem 1.2.1]{DupuisEllis} we obtain {an alternative formulation of the} large deviations principle for the Circular beta ensemble at high temperature with continuous potential, see Theorem \ref{thm:alternative LDP} below.}

{
To achive our goal, we need several technical results regarding the distribution $\Theta_\nu$ \eqref{def theta}, and the CMV matrix $E$ \eqref{E}.
First, in the next Lemma, we {give a reprensentation } of $\Theta_\nu$ in terms of Gaussian, and Chi distributions. 

\begin{Lemma}
\label{lem_representation}
Let $\nu>1$. Let $X_1,X_2,Y_\nu$ be independent, $X_1$, $X_2$ standard Gaussian variables and $Y_\nu$ be $\chi_{\nu-1}$ distributed, i.e. with density $$\chi_{\nu-1}(x)=\frac{2^\frac{3-\nu}{2}}{\Gamma(\frac{\nu-1}{2})}x^{\nu-2}e^{-x^2/2}\mathbf{1}_{x>0}\,,$$
here $\Gamma(x)$ is the classical Gamma function \cite[§5]{DLMF}
\begin{equation}
    \label{eq gamma}
    \Gamma(x)=\int_0^{+\infty}t^{x-1}e^{-t}dt\,,
\end{equation}
Then, $Z:=\frac{X_1+iX_2}{(X_1^2+X_2^2+Y_\nu^2)^{1/2}}$ follows the law $\Theta_\nu$.
\end{Lemma}

\begin{remark}
    If $\nu\geq 2$ is an integer, $Y_\nu^2$ has the law of $\sum_{i=1}^{\nu-1}N_i^2$ where the $N_i$'s are i.i.d. standard gaussians random variables, thus $Z$ is equal in distribution to $\frac{X_1+iX_2}{(X_1^2+\dots+X_{\nu+1}^2)^\frac{1}{2}}$, which follows the law $\Theta_{\nu}$ by Remark \ref{rem: interpretation}.
\end{remark}

\begin{proof}
We identify $\C$ with $\R^2$ and check that for any $f:\D \to \R$ bounded and measurable, 
$$  \E[f(Z)]= \frac{\nu-1}{2\pi}\int_{\D} f(z)(1-|z|^2)^\frac{\nu-3}{2} \di^2z,$$
i.e. that for some constant $c$,

\begin{equation}
    \begin{split}
\int_{\R^2\times \R_+^*} f\left( \frac{x_1}{(x_1^2+x_2^2+y^2)^{1/2}}, \frac{x_2}{(x_1^2+x_2^2+y^2)^{1/2}}\right)&e^{-\frac{x_1^2+x_2^2}{2}}e^{-\frac{y^2}{2}}y^{\nu-2}\di x_1\di x_2\di y \\ &
= c\int_{\D} f(u,v)(1-(u^2+v^2))^\frac{\nu-3}{2} \di u\di v\,.
    \end{split}
\end{equation}

For fixed $y>0$, we perform the diffeomorphic change of variables $(u,v)=\frac{1}{(x_1^2+x_2^2+y^2)^{1/2}}(x_1,x_2)$. Its inverse $(x_1,x_2)=\frac{y}{(1-(u^2+v^2))^{1/2}}(u,v)$ has Jacobian equal to $y^2(1-(u^2+v^2))^{-2}$. The integral becomes
\begin{equation}
    \int_\D\frac{f(u,v)}{(1-(u^2+v^2))^{2}}\int_{\R_+^*} e^{-\frac{y^2}{2(1-(u^2+v^2))}}y^\nu\di y\di u\di v=\sqrt{\frac{\pi}{2}}\int_\D\frac{f(u,v)}{(1-(u^2+v^2))^{3/2}}\E[|X_{u,v}|^\nu]\di u \di v\,,
\end{equation} 
where $X_{u,v}$ denotes a Gaussian variable $\cN(0,1-(u^2+v^2))$. By \cite{GaussianMoments},
$$\E[|X_{u,v}|^\nu] = c_\nu(1-(u^2+v^2))^{\frac{\nu}{2}}$$
for some constant $c_\nu$ independent of $u,v$. Substituting this last equality in the previous integral, we conclude.

\end{proof}

To obtain our main result, we need some technical lemmas. Since they are based on standard techniques, we just state them here, and we defer their proofs to the appendix \ref{app:A}.

The first one gives an estimate that we use combined with Lemma \ref{LEM:DISTANCE_INEQ}.

}

\begin{Lemma}
\label{lem:estimate_product}
Let $N=2k$ be even and $A$ be a $N\times N$ matrix. Then, \begin{itemize}
    \item $\sum_{i,j=1}^N|(\cL A)_{i,j}| \leq 2\sum_{i,j=1}^N|A_{i,j}|$,
    \item $\sum_{i,j=1}^N|(A \cM)_{i,j}| \leq 2\sum_{i,j=1}^N|A_{i,j}|,$
\end{itemize}
{where $\cM$, and $\cL$ are defined in \eqref{eq:cM_cL}.}
\end{Lemma}

{We now give an explicit coupling between $\Theta_\nu$ \eqref{def theta} and $\Theta_{\nu+h}$ for $\nu>1$, $h>0$. \\
Let $X_1, X_2$ be $\cN(0,1)$ independent variables, and let $Y_{\nu-1}\sim \chi_{\nu-1}$, $Y_h\sim \chi_h$ be independent, and independent of $X_1,X_2$ (notice that $(Y_h^2+Y_{\nu-1}^2)^{\frac{1}{2}}$ is $\chi_{\nu+h-1}$ distributed).
Let 
\begin{equation}
 \label{eq_coupling}
     \alpha_\nu = \frac{X_1+iX_2}{(X_1^2+X_2^2 + Y_{\nu-1}^2)^{\frac{1}{2}}},\ \ \ \ \ \ \alpha_{\nu+h}=\frac{X_1+iX_2}{(X_1^2+X_2^2 + Y_{\nu-1}^2+Y_h^2)^{\frac{1}{2}}}.
 \end{equation}
 By Lemma \ref{lem_representation}, $\alpha_\nu\sim\Theta_\nu$ and $\alpha_{\nu+h}\sim\Theta_{\nu+h}$.\\
 
Exploiting this coupling, we bound the differences $|\alpha_\nu - \alpha_{\nu+h}|$ and $|\rho_\nu - \rho_{\nu+h}|$ by a random variable $Z_h$, where $\rho_\nu = \sqrt{1-|\alpha_\nu|^2}$, and $\rho_{\nu+h} = \sqrt{1-|\alpha_{\nu+h}|^2}$. Moreover, we find an upper bound for the exponential moments of $Z_h$.
 
 \begin{Lemma}
\label{lem:bomb}
Let $\alpha_\nu $ and $\alpha_{\nu+h}$ defined by equation \eqref{eq_coupling}. Define $\rho_\nu = \sqrt{1-|\alpha_\nu|^2}$, and $\rho_{\nu+h} = \sqrt{1-|\alpha_{\nu+h}|^2}$, then the following holds

\begin{enumerate}[i.]
    \item \begin{equation}
        \begin{split}
            &|\alpha_\nu - \alpha_{\nu+h}| \leq \frac{Y_h}{(X_1^2 + X_2^2+Y_h^2)^\frac{1}{2}}\,\quad \text{almost surely}\,,\\
            &|\rho_\nu - \rho_{\nu+h}| \leq \frac{Y_h}{(X_1^2 + X_2^2+Y_h^2)^\frac{1}{2}}\,\quad \text{almost surely},
        \end{split}
    \end{equation}
    where $X_1,X_2\sim \cN(0,1)$, $Y_h\sim \chi_h$ are all independent.
    
    \item define $Z_h = \frac{Y_h}{(X_1^2 + X_2^2+Y_h^2)^\frac{1}{2}}$, and $a(h) = -\frac{1}{2}\ln(h) +1$, then there exists a constant $K$ independent of $h$ such that
    \begin{equation}
    \label{eq:sup_bound}
    \sup_{0<h<1}\meanval{\exp(a(h) Z_h)} \leq K\,.
\end{equation}
\end{enumerate}
\end{Lemma}}

\begin{remark}
\label{rem:monotone coupling}
    Let $h<h'$, and let $Z_h,Z_{h'}$ be given by
    $$Z_h = \frac{Y_h}{(X_1^2+X_2^2+Y_h^2)^{\frac{1}{2}}}, \qquad Z_{h'} = \frac{Y_{h'}}{(X_1^2+X_2^2+Y_{h'}^2)^{\frac{1}{2}}}\, ,$$
    where $Y_h\sim \chi_h$ and $Y_{h'}\sim \chi_{h'}$ are $\chi$ variables coupled by
    $$ Y_{h'} = \sqrt{Y_h^2 + Z^2},$$
    $Z$ being a $\chi_{h'-h}$ variable independent of $Y_h$. Then, because of the monotonicity of the function $x\mapsto \frac{x}{\sqrt{a+x^2}}$ for $a>0$, we have almost surely $Z_h \leq Z_{h'}$.
\end{remark}

{
We are now in position { to give an alternate formulation of the large deviations principle for the sequence of measures ${(\mu_N(E))}$ under the law $\mathbb{P}^{\frac{2\beta}{N}}_{C,N}$, given by Theorem \ref{thm_LDPCoulomb}}. 
\begin{Lemma}
Let $\beta>0$. The law of the empirical measure ${(\mu_N(E))_{N \text{even}}}$ under $\mathbb{P}_{C,N}^{\frac{2\beta}{N}}$ satisfies a large deviations principle {at speed} $N$ and with a good rate function

\begin{equation}
    I_\beta(\mu) = \lim_{\delta\to 0}\liminf_{q\to \infty} \inf_{\nu_{\beta/q},\ldots, \nu_\beta\atop\frac{1}{q}\sum_i \nu_{i\beta/q} \in B_\mu(\delta)} \left\{ \frac{1}{q} \sum_{i=1}^q J_{i\beta/q}(\nu_{i\beta/q})\right\}\,,
\end{equation}
where $J_\eta$ is the rate function of Proposition \ref{prop large deviations} applied to $\sigma = \Theta_{2\eta+1}$. %{and $N,q$ are related by $N=qk+r$, where $\N \ni r<q$, and $k\in \N$ .}
\end{Lemma}

\begin{proof}
Following the same line as in \cite[Lemma 3.3]{GMToda}, we proceed by exponential approximation. Let $q\geq 1$ be an integer, since $N$ is even, we can write $N = kq + r$, with $k$ even, and with $0\leq r < 2q-2$. Consider the following family of matrices $\cL^{(i)},\cM^{(i)}$, $i=1,\ldots q$ defined as 

	\begin{equation}
		\cM_k^{(i)}= \begin{pmatrix}
			-\alpha_{ik}&&&&& \rho_{ik} \\
			& \Xi^{(i)}_2 \\
			&& \Xi^{(i)}_4 \\
			&&& \ddots \\
			&&&&\Xi^{(i)}_{k-2}\\
			\rho_{ik} &&&&& \wo \alpha_{ik}
		\end{pmatrix}\, ,\qquad 
		\cL_k^{(i)} = \begin{pmatrix}
			\Xi^{(i)}_{1} \\
			& \Xi^{(i)}_3 \\
			&& \ddots \\
			&&&\Xi^{(i)}_{k-1}
		\end{pmatrix} \,,
	\end{equation}
where $\Xi^{(i)}_\ell$ are defined as

	\begin{equation}
	\label{eq: local xi}
		\Xi^{(i)}_\ell(\alpha_{i(k-1) + \ell})=\Xi^{(i)}_\ell = \begin{pmatrix}
			\wo \alpha_{i(k-1) + \ell} & \rho_{i(k-1) + \ell} \\
			\rho_{i(k-1) + \ell} & -\alpha_{i(k-1) + \ell}
		\end{pmatrix}\,, \quad \rho_{i(k-1) + \ell} = \sqrt{1-|\alpha_{i(k-1) + \ell}|^2} ,\quad \ell=1,\dots, k-1\, ,
	\end{equation}
and $(\alpha_{i(k-1)+\ell})_{1\leq i \leq q, 1\leq \ell \leq k}$ is a family of independent random variables such that $$\alpha_{i(k-1)+\ell}\sim \Theta_{2\beta\frac{N-ik}{N} + 1},\, \ell=1,\ldots k, \, i = 1,\ldots,q.$$
From these two families of matrices, we can define a third one, namely $\cE_k^{(i)} = \cL_k^{(i)}\cM_k^{(i)},\, i=1,\ldots,q$. We notice that $\cE_k^{(i)}$ is distributed according to $\mathbb{P}_{AL,k}^{\beta\frac{N-ik}{N}}$, and that {the $\cE_k^{(i)}$, $1\leq i \leq q$,} are independent.

Our aim is to prove that the empirical measure of the matrix $C_N^q$

\begin{equation}
    C_N^q = \begin{pmatrix} \cE_{k}^{(1)} \\
    & \cE_{k}^{(2)}\\
    && \ddots \\
    &&& \cE_{k}^{(q)}\\
    &&&& 0_{r\times r}\,
    \end{pmatrix},
\end{equation}
where $0_{r\times r}$ is a null block of size $r\times r$, is an exponential approximation (see \cite[Definition 4.2.14]{De-Ze}) of {the empirical measure of} $E\sim \mathbb{P}_{C;N}^{\frac{2\beta}{N}}$ \eqref{E}, {that is,} for any positive real number $\delta$:

\begin{equation}
\label{eq:real claim}
    \lim_{q\to \infty} \limsup_{N\to \infty} \frac{1}{N}{\ln\left(\mathbb{P}\left( d\left({\mu_N(E)}, { \mu_N(C_N^q)}\right)\right) > \delta \right)} = -\infty\,,
\end{equation}

where $\mathbb{P}$ denotes the coupling introduced in equation \eqref{eq_coupling}. In this way, we obtain the claim as an application of \cite[Theorem 4.2.16]{De-Ze}. The strategy of proof is the following. First we approximate $C_N^q$ and $E$ by two block diagonal matrices $\wt C_N^q, \wt E_N^q$ respectively. Finally, we will prove that both $\wt C_N^q$, and  $\wt E_N^q$ approximate a third matrix $B_N^q$.

Consider another family of matrices $(\wt \cM_k^{(i)})_{1\leq i\leq q}$ of size $k\times k$, defined as 

	\begin{equation}
 \wt \cM_k^{(i)} = \mbox{diag}\left(\wt \Xi^{(i)}_{0},\Xi^{(i)}_2,\Xi^{(i)}_4, \ldots, \wt \Xi^{(i)}_k\right)\,,
	\end{equation}

	where the matrices $\Xi^{(i)}_\ell $ are defined in \eqref{eq: local xi},
% 	for $1\leq i \leq q$ the block $\Xi^{(i)}_j$, $j={0},\dots, k-1$, takes the form 
% 	\begin{equation}
% 		\label{eq:xi_def}
% 		\Xi^{(i)}_j = \begin{pmatrix}
% 			\wo \alpha_{(i-1)k+j} & \rho_{(i-1)k+j} \\
% 			\rho_{(i-1)k+j} & -\alpha_{(i-1)k+j}
% 		\end{pmatrix}\, ,\;\;\rho_{(i-1)k+j} = \sqrt{1-|\alpha_{(i-1)k+j}|^2}, 
% 	\end{equation}	
	while $\wt \Xi^{(i)}_{0} = (1)$ and $\wt \Xi^{(i)}_{k} = (\wo \alpha_{ik})$ are $1\times 1$ matrices, where the $\alpha_{ik}$ are independent, uniformly distributed on the unit circle for all $i=1,\ldots,q$, and independent from $(\alpha_{i(k-1)+l})_{1\leq i \leq q, 1\leq \ell \leq k-1}$. Define the $k \times k$ family of CMV matrices
	$$\wt \cE_k^{(i)} =  \cL_k^{(i)}\wt \cM_k^{(i)}\,\quad i = 1, \ldots, q\, .$$
	\begin{comment}
	and suppose that the entries $\alpha_{(i-1)k+\ell} \sim \Theta_{2\beta\frac{N-ik}{N}}$ for all $\ell=1, \ldots, k-1$, and $i=1,\ldots,q$, and $\alpha_{ik}$ to be uniformly distributed on the unit circle for all $i=1,\ldots,q$. 
	\end{comment}
	From the family of matrices $(\wt \cE_k^{(i)})_{1\leq i \leq q}$, we define the block diagonal matrix:
	
\begin{equation}
    \wt C_N^q = \begin{pmatrix}\wt \cE_{k}^{(1)} \\
    & \wt \cE_{k}^{(2)}\\
    && \ddots \\
    &&& \wt \cE_{k}^{(q)}\\
    &&&& 0_{r\times r}
    \end{pmatrix}\, .
\end{equation}
	We claim that $\wt C^q_N$ is such that
	\begin{equation}
	\label{eq:rank_1}
	     \mathrm{rank}(C^q_N - \wt C_N^q) \leq 2q\,.
	\end{equation}
	Indeed, we take the same $\alpha^{(i)}_j$ in the construction of $\wt \cE_k^{i}$ and of $\cE_k^{i}$, except for the entries of the corners of $\cM_k^{(i)}$, where $\cM_k^{(i)}(1,1)$ is replaced by $1$, $\cM_k^{(i)}(k,k)$ is replaced by a uniform variable on the circle, and both entries $\cM_k^{(i)}(1,k)$ and $\cM_k^{(i)}(k,1)$ are replaced by $0$. This shows that  
	$$ \text{rank}(\cM_k^{(i)} - \wt \cM_k^{(i)}) \leq 2,$$ 
	and
	$$ \text{rank}(\cE_k^{(i)} - \wt \cE_k^{(i)}) = \text{rank}(\cL_k^{(i)} (\cM_k^{(i)} - \wt \cM_k^{(i)})) \leq \text{rank}(\cM_k^{(i)} - \wt \cM_k^{(i)}) \leq 2, $$
	and we deduce \eqref{eq:rank_1}.
	\begin{comment}
	Indeed it follows easily from the definition of $\wt \cE_k^{i}$ that it is a rank $2$ perturbation of  $\cE_k^{i}$, thus we deduce the previous inequality.
	\end{comment}
	 From \eqref{eq:rank_1} and  Lemma \ref{LEM:DISTANCE_INEQ}, we deduce that

\begin{equation}
    d({\mu_N(C_N^q)}, {\mu_N(\wt C_N^q)} ) \leq \frac{2}{k}\,, 
\end{equation}
and for any ${\delta >0}$ and sufficiently large $N$, we can take $k$ such that $\frac{2}{k}\leq \frac{\delta}{4}$.

Consider now another two families of matrices $( \fL_k^{(i)})_{1\leq i\leq q}$, and $(\fM_k^{(i)})_{1\leq i\leq q}$, constructed in the same way as $( \cL_k^{(i)})_{1\leq i\leq q}$, and $(\wt \cM_k^{(i)})_{1\leq i\leq q}$ by means of independent variables $\wt \alpha_{(i-1)k+j}$, where each $\wt\alpha_{(i-1)k+j} \sim \Theta_{2\beta\frac{N-(i-1)k-j}{N}}$ is coupled to $\alpha_{(i-1)k+j}$ by equation \eqref{eq_coupling}, for all $j=0, \ldots, k-1$, and $i=1,\ldots,q$, and where $\wt \alpha_{ik}=\alpha_{ik}$ for $i=1,\ldots,q$ is uniformly distributed on the unit circle. Define the family of CMV matrices $(\fE_k^{(i)})_{1\leq i\leq q}$ as

%Consider now another two families of matrices $(\wt \fL_k^{(i)})_{1\leq i\leq q}$, and $(\wt \fM_k^{(i)})_{1\leq i\leq q}$, constructed in the same way as $( \cL_k^{(i)})_{1\leq i\leq q}$, and $(\wt \cM_k^{(i)})_{1\leq i\leq q}$, but in this case $\alpha_{(i-1)k+j} \sim \Theta_{2\beta\frac{N-(i-1)k-j}{N}}$ for all $j=0, \ldots, k-1$, and $i=1,\ldots,q$, and $\alpha_{ik}$ to be uniformly distributed on the unit circle for all $i=1,\ldots,q$. Define the family of CMV matrices $(\fE_k^{(i)})_{1\leq i\leq q}$ as

\begin{equation}
    \fE_k^{(i)} = \fL_k^{(i)} \fM_k^{(i)}\, ,\quad i = 1,\ldots,q \,. 
\end{equation}
Define the block diagonal matrix $E_N^q$ as:

\begin{equation}
     E_N^q = \begin{pmatrix} \fE_{k}^{(1)} \\
    &  \fE_{k}^{(2)}\\
    && \ddots \\
    &&&  \fE_{k}^{(q)}\\
    &&&& 0_{r\times r}\,.
    \end{pmatrix}
\end{equation}
From the definition of ${\mathbb{P}^{\frac{2\beta}{N}}_{C,N}}$ and $E_N^q$, we conclude that for some $E\sim \mathbb{P}^{\frac{2\beta}{N}}_{C,N} $, we have
\begin{equation}
    \mathrm{rank}(E-E_N^q) \leq 2q + r\,.
\end{equation}
As before, from the previous inequality we deduce that
% We claim, in the same fashion as for equation \eqref{eq:rank_1}, that we can choose the entries of the matrices $E$ and $\wt E_N^q$ {in such a way that} 
% \begin{equation}
%     \mathrm{Rank}(E-\wt E_N^q) \leq 2q + r\,.
% \end{equation}
% Indeed this follows easily from the definition of $E$, and $\wt E_N^q$. 

\begin{equation}
    d(\mu_N(E), \mu_N(\wt E_N^q) ) \leq \frac{4}{k}\,.
\end{equation}

Finally, we define the matrix $B_N^q$ as

\begin{equation}
     B_N^q = \begin{pmatrix} B_{k}^{(1)} \\
    &  B_{k}^{(2)}\\
    && \ddots \\
    &&&  B_{k}^{(q)}\\
    &&&& 0_{r\times r}\,,
    \end{pmatrix}
\end{equation}
where $B_k^{(i)} = \fL_k^{(i)}\wt \cM_k^{(i)}$.

Let $\delta > 0$, for $N$ large enough such that $\frac{4}{k}\leq \frac{\delta}{4}$, we have almost surely 
$$d(\mu_N(C_N^q), \mu_N(\wt C_N^q) ) +d(\mu_N(E),  \mu_N(E_N^q) ) > \frac{\delta}{2}. $$
As a consequence,

\begin{equation}
    \begin{split}
       & \mathbb{P}\left( d(\mu_N(C_N^q), \mu_N(E)) > \delta \right) \\ &\leq     \mathbb{P}\left( d(\mu_N(C_N^q), \mu_N(\wt C_N^q)) + d(\mu_N(\wt C_N^q), \mu_N(B_N^q)) + d(\mu_N(B_N^q), \mu_N( E_N^q)) + d(\mu_N(E_N^q),\mu_N(E)) > \delta \right)\\
       &\leq \mathbb{P}\left(  d(\mu_N(\wt C_N^q), \mu_N(B_N^q)) + d(\mu_N(B_N^q), \mu_N(E_N^q))   > \frac{\delta}{2} \right)\,.
    \end{split}
\end{equation}
Moreover, combining Lemma \ref{LEM:DISTANCE_INEQ} and Lemma \ref{lem:estimate_product} we deduce that

\begin{equation}
    \begin{split}
        &d(\mu_N(\wt C_N^q), \mu_N(B_N^q)) \leq \frac{2}{N} \sum_{i=1}^q \sum_{1\leq \ell,j\leq k} |\fL^{(i)}_k(\ell,j) - \cL^{(i)}_k(\ell,j)|\,, \\
        &d(\mu_N(B_N^q), \mu_N(\wt E_N^q)) \leq \frac{2}{N} \sum_{i=1}^q \sum_{1\leq \ell,j\leq k} |\fM^{(i)}_k(\ell,j) - \wt\cM^{(i)}_k(\ell,j)|\,.
    \end{split}
\end{equation}
Applying Lemma \ref{lem:bomb} point $i.$, we deduce that

\begin{equation}
    d(\mu_N(\wt C_N^q), \mu_N(B_N^q)) + d(\mu_N(B_N^q), \mu_N(\wt E_N^q)) \leq \frac{8}{N}\sum_{i=1}^{q} \sum_{j=0}^{k-1} Z^{(i)}_{\frac{k-j}{N}}\,,
\end{equation}
where the last sum denotes the sum of independent random variables with law $Z_{\frac{k-j}{N}}$, defined in Lemma \ref{lem:bomb}.

Thus, for $N$ large enough such that $\frac{4}{k}\leq \frac{\delta}{4}$, we deduce that for any non-negative function $a(q^{-1})$:

\begin{equation}
    \begin{split}
        \mathbb{P}\left( d(\mu_N(C_N^q), \mu_N(E)) > \delta \right)& \leq \mathbb{P}\left( \sum_{i=1}^{q} \sum_{j=0}^{k-1} Z^{(i)}_{\frac{k-j}{N}} > \frac{N\delta}{16} \right)\\
        & \leq e^{-a(q^{-1})N\delta/16}\left(\sup_{0<h<1}\meanval{\exp(a(h) Z_h)} \right)^{qk}\,,
    \end{split}
\end{equation}
Where in the last inequality we used Remark \ref{rem:monotone coupling}, {namely, since} $\frac{k-j}{N}\leq \frac{1}{q}$, we have
$$ \E[\exp(a(q^{-1})Z_{\frac{k-j}{N}} ] \leq \E[\exp(a(q^{-1})Z_{\frac{1}{q}} ]. $$

Setting $a(h) = -\frac{1}{2}\ln(h) + 1$ and applying Lemma \ref{lem:bomb} point $ii.$, we deduce that there exist constants $\wt K$ and $c>0$, independent of $q$, such that

\begin{equation}
    \frac{1}{N}\ln\left(\mathbb{P}\left( d(\mu_N(C_N^q), \mu_N(E)) > \delta \right)\right) \leq -c\ln(q)\delta + \wt K\,,
\end{equation}
And we obtain the claim.
\end{proof}

We can apply the previous Lemma to study the case of continuous potential, indeed as a consequence of Varadhan's Lemma we obtain the main result of this section, namely

\begin{theorem}
\label{thm:alternative LDP}
    In the same notation as before. Let $\beta>0$, and  $V\; :\; \T \to \R$ continuous. The law of the empirical measures $ \mu_N(E)$
    under $\di \mathbb{P}^{V,\frac{2\beta}{N}}_{C,N}$ satisfies a large deviations principle {at speed} $N$, with a good rate function $I_\beta^V(\mu) = f_\beta^V(\mu) - \inf_{\nu \in \cP(\T)} f_\beta^V(\nu)$, where 
    
    \begin{equation}
    \label{eq: rate function C}
       f_\beta^V(\mu) =  \lim_{\delta\to 0}\liminf_{q\to \infty} \inf_{\nu_{\beta/q},\ldots, \nu_\beta\atop\frac{1}{q}\sum_i \nu_{i\beta/q} \in B_\mu(\delta)} \left\{ \frac{1}{q} \sum_{i=1}^q\left( J_{i\beta/q}(\nu_{i\beta/q}) +  \int_\T V\di \nu_{i\beta/M}\right)\right\}\,,
    \end{equation}
    %{here $N,q$ are related by $N=qk+r$, where $\N \ni r<q$, and $k\in \N$ .}
\end{theorem}

\section{Proof of the main results}
\label{sec main thm}

In this section, we conclude the proof of Theorem \ref{THM:MAIN_1} and prove Theorem \ref{THM:FINAL_RELATION}. The main tool to prove these theorems is the uniqueness of the minimizer of the rate function for the $\beta$ ensemble.

Define the free energies of the Ablowtiz-Ladik lattice and the Circular beta ensemble at high temperature as

\begin{equation}
    F_{AL}(V,\beta) = \inf_{\nu \in \cP(\T)} g_\beta^V(\nu)\,, \quad F_C(V,\beta) =  \inf_{\nu\in\cP(\T)} f_\beta^V(\nu)\,,
\end{equation}
where $g_\beta^V$, and $f_\beta^V$ are given by \eqref{eq: rate function AL} and \eqref{eq: rate function C}.
We claim that

\begin{Lemma}
\label{lem_conclusion}
{Let $\beta >0$, and $V\; :\; \T \to \R$ continuous}, then the following holds:

\begin{enumerate}[a.]
    \item the map $\beta\to F_C(V,\beta)$ is continuously differentiable on $\R_+^*$. Moreover:
    \begin{equation}
        F_{AL}(V,\beta) = \partial_\beta(\beta F_C(V,\beta))\,;
    \end{equation}
    \item for almost all $\beta>0$ there exists a unique minimizer $\nu_\beta^V$ of the functional $J_\beta^V(\mu)$, see Corollary \ref{cor: LDP AL}, given by
    
    \begin{equation}
        \nu_\beta^V = \partial_\beta (\beta \mu_\beta^V)\,,
        \end{equation}
    {i.e for continuous $f:\T\to \R$,
    \begin{equation}
        \int_\T f \di\nu_\beta^V = \partial_\beta \left(\beta \int_\T f \di\mu_\beta^V \right)\,.
        \end{equation}}
        we recall that the measure $\mu_\beta^V$ is defined as the unique minimizer of the functional $I_\beta^V$ in Theorem \ref{thm:alternative LDP}.
\end{enumerate}
\end{Lemma}

\begin{remark}
    Our definition of Free Energy is different from the one used in \cite{mazzuca2021generalized,spohn2021hydrodynamic}. Indeed,  in virtue of Varadhan's Lemma \cite[Theorem 1.2.1]{DupuisEllis}, we have
    \begin{equation}
        \label{eq:equality free energy}
        \begin{split}
                    &F_{AL}(V,\beta)=\inf_{\nu \in \cP(\T)} g_\beta^V(\nu) = -\lim_N \frac{1}{N}\ln\E_{\text{AL},N}^\beta\left[e^{-\Tr{V(\cE)}} \right],\\ & F_{C}(V,\beta) = \inf_{\nu \in \cP(\T)} f_\beta^V(\nu) =- \lim_N \frac{1}{N}\ln\E_{\text{C},N}^\beta\left[e^{-\Tr{V(E)}} \right],
        \end{split}
    \end{equation}
    instead in \cite{mazzuca2021generalized,spohn2021hydrodynamic}, the authors defined the free energies as
    \begin{equation}
        \begin{split}
           & \wt F_{AL}(V,\beta) = -\lim_{N\to \infty}\frac{1}{N}\ln(Z_N^{AL}(V,\beta))\,,\\
           & \wt F_{C}(V,\beta) = -\lim_{N\to \infty}\frac{1}{N}\ln(Z_N^{C}(V,\beta))\,.
        \end{split}
    \end{equation}
 We notice that it is possible to recover one expression from the other since 
 \begin{equation}
     \begin{split}
         &F_{AL}(V,\beta) = \wt F_{AL}(V,\beta) - \wt F_{AL}(0,\beta)\,, \\
         & F_{C}(V,\beta) = \wt F_{C}(V,\beta) - \wt F_{C}(0,\beta)\,.
     \end{split}
 \end{equation}
\end{remark}

To prove uniqueness of the minimizer $\nu_\beta^V$, we need to consider a continuous family $(\mu^*_s)_{0<s<\beta}$, where each $\mu^*_s$ minimizes $J_s^V$, see Corollary \ref{cor: LDP AL}.
We address the existence of such a family in the next Lemma, which we prove in the appendix \ref{app:A}.
\begin{Lemma}
\label{Lemma_minimizers}
Let $M_\beta^V = (J_\beta^V)^{-1}(\{0\})$ be the set of minimizers of $J_\beta^V$, defined in Corollary \ref{cor: LDP AL}. Then, $\beta \mapsto M_\beta^V$ is continuous in the sense that for all $\varepsilon>0$, there exists $\delta>0$ such that for all $0\leq h \leq \delta$, $M_{\beta +h}^V \subset (M_\beta^V)^\varepsilon$, where for $A\subset \mathcal{P}(\T)$ we denote $A^\varepsilon = \{\mu\in \mathcal{P}(\T)\ |\ d(\mu,A)\leq \varepsilon \}.$
\end{Lemma}

\begin{proof}[Proof of Lemma \ref{lem_conclusion}]
First, we notice that for any probability measure $\mu\in \cP(\T)$, Theorem \ref{thm:alternative LDP} implies

\begin{equation}
\label{eq:lb_energy}
\begin{split}
    f_\beta^V(\mu) &\geq \liminf_{{q}\to \infty} \inf_{\nu \in \cP(\T)} \left\{ \frac{1}{{q}} \sum_{i=1}^{q}\left( J_{i\beta/{q}}(\nu) +  \int_\T V\di \nu\right)\right\}\\
    & = \int_0^1 \inf_{\nu \in \cP(\T)} g_{s\beta}^V(\nu) =  \int_0^1 F_{AL}(V,s\beta)\di s\,,
    \end{split}
\end{equation}
Where we noticed that the Riemann sums indeed converge towards the integral since $s\mapsto F_\text{AL}(V,s\beta)$ is concave, this can be seen by applying Hölder inequality to equation \eqref{eq:equality free energy}.

To prove the first part of the claim, we show that the lower bound is achieved. For $s\in[0,1]$, let $\nu_{s\beta}^*$ be {a} minimizer of $\inf_{\nu\in \cP(\T)}g_{s\beta}^V(\nu)$. From Lemma \ref{Lemma_minimizers}, we can choose $\nu^*_{s\beta}$ such that the map $s \to \nu^*_{s\beta} $ is continuous. This implies that $\mu^*_\beta = \int_{0}^1 \nu^*_{s\beta}\di s$ is a {well-defined} probability measure on $\T$. We claim that this measure minimizes $f_\beta^V$  \eqref{eq: rate function C}, and so $I_\beta^V$. Indeed, from Theorem \ref{thm:alternative LDP}, we deduce that

\begin{equation}
\label{eq:ub_energy}
\begin{split}
    f_\beta^V(\mu_\beta^*) & =   \lim_{\delta\to 0}\liminf_{q\to \infty} \inf_{\nu_{\beta/q},\ldots, \nu_\beta\atop\frac{1}{q}\sum_i \nu_{i\beta/q} \in B_{\mu_\beta^*}(\delta)} \left\{ \frac{1}{q} \sum_{i=1}^q\left( J_{i\beta/q}(\nu_{i\beta/q}) +  \int_\T V\di \nu_{i\beta/q}\right)\right\} \\
    & \leq \liminf_{q\to \infty} \left\{ \frac{1}{q} \sum_{i=1}^q\left( J_{i\beta/q}(\nu^*_{i\beta/q}) +  \int_\T V\di \nu^*_{i\beta/q}\right)\right\}\\
    & ={ \liminf_{q\to \infty} \left\{ \frac{1}{q} \sum_{i=1}^q\inf_{\nu \in \cP(\T)}\left( J_{i\beta/q}(\nu) +  \int_\T V\di \nu\right)\right\} } \\
    & = \int_0^1 \inf_{\nu \in \cP(\T)} g_{s\beta}^V(\nu) =  \int_0^1 F_{AL}(V,s\beta)\di s\,.
\end{split}
\end{equation}

Combining \eqref{eq:lb_energy}-\eqref{eq:ub_energy}, and performing the change of coordinates $s\beta = t$ we deduce that:
\begin{equation}
    \beta F_{C}(V,\beta) = \int_0^\beta F_{AL}(V,t) \di t\,. 
\end{equation}
Moreover, from Lemma \ref{lem: prop circ} we deduce that the map $\beta \to F_C(V,\beta)$ is Lipschitz in $\beta$, and so almost surely differentiable. This implies that for almost all $\beta >0$

\begin{equation}
    F_{AL}(V,\beta) = \partial_\beta(\beta F_C(V,\beta))\,.
\end{equation}

Furthermore, we have just shown that $I_\beta^V(\mu) = f_\beta^V(\mu) - \inf_{\nu\in \cP(\T)}f_\beta^V(\nu)$ reaches its minimum at $\int_0^1 \nu^*_{s\beta}\di s$. {By uniqueness of the minimizer of $I_\beta^V(\mu)$, Theorem \ref{thm_LDPCoulomb}, we deduce that we have the equality between probability measures $\mu^V_\beta=\int_0^1 \nu^*_{s\beta}\di s$. Taking $f:\T\to \R$ continuous we get
$$ \beta\int_\T f\di \mu^V_\beta = \int_0^\beta \int_\T f\di \nu^*_{s}\di s.  $$
Note that the function $s\mapsto \int_\T f\di \nu^*_{s}$ is continuous, therefore by differentiating this equality, we get that $\nu^*_\beta$ is the unique minimizer of $J^V_\beta$, which we denote by $\nu^V_\beta$, and satisfies for $f$ continuous
\begin{equation}
    \int_\T f \nu^V_\beta = \partial_\beta\left(\beta \int_\T f\mu_\beta^V\right)\,,
\end{equation}
proving point b.
}

\begin{comment}
This, and the previous proof, implies that the minimizer $\nu_\beta^V$ of $g_\beta^V(\nu)$ is explicitly given by $\nu^*_\beta = \nu^V_\beta$. From this equality we deduce that

\begin{equation}
    \nu^V_\beta = \partial_\beta(\beta \mu_\beta^V)\,.
\end{equation}
Moreover, by corollary \ref{cor: LDP AL} the map $\beta \to \nu_\beta^V$ is continuous, and hence the measure $\nu_\beta^V$ is defined everywhere.
\end{comment}

\end{proof}

\begin{remark}
    As a corollary of the previous Lemma we obtain Theorem \ref{THM:MAIN_1}.
\end{remark}

\section[Schur flow]{The Schur Flow}
\label{sec: Schur}
In this section, we consider another integrable model, namely the Schur flow. Our goal is to show that is possible to obtain a similar result to the one that we presented for the Ablowitz-ladik lattice. Namely, we prove the existence of a large deviations principle for the Schur flow, and we relate its density of state to the one of the Jacobi beta ensemble in the high temperature regime.

\subsection[GGE]{Generalized Gibbs Ensemble}
{
The Schur flow is the system of ODEs \cite{Golinskii}

\begin{equation}
\label{eq:schur}
    \dot{\alpha_{j}} = \rho_j^2(\alpha_{j+1}-\alpha_{j-1})\,, \quad \rho_j = \sqrt{1-|\alpha_j|^2}
\end{equation}
and, as before, we consider periodic boundary conditions, namely $\alpha_j = \alpha_{j+N}$ for all $j\in \Z$.

 In \cite{Ablowitz1975}, it is argued that the continuum limit of \eqref{eq:schur} is the modified Korteweg-de Vries equation:

\begin{equation}
    \partial_t u = \partial_x^3 u -6u^2\partial_x u\,.
\end{equation}

We notice that, if one chooses an initial data such that $\alpha_j(0)\in\R$ for all $j=1,\ldots,N$, then $\alpha_j(t)\in\R$ for all times. Moreover, it is straightforeward to verify that $K_0 = \prod_{j=1}^N \left( 1-|\alpha_j|^2\right)$ is conserved along the Schur flow. This implies that we can choose as phase space for the Schur flow the $N$-cube $\mathbb{I}^N$, where $\I := (-1,1)$.

On this phase space, we consider the Poisson braket \eqref{eq:poisson_bracket}, so we can rewrite the Schur flow \eqref{eq:schur} in Hamiltonian form as

\begin{equation}
    \dot \alpha_j = \{\alpha_j, H_S\},\quad H_S = - i\sum_{j=1}^N \left( \alpha_{j}\wo \alpha_{j+1} - \wo \alpha_j\alpha_{j+1}\right)\,.
\end{equation}}
% }
% The Hamiltonian of the Schur flow reads:

% \begin{equation}
%     H_{S} = - i\sum_{j=1}^N \left( \alpha_{j}\wo \alpha_{j+1} - \wo \alpha_j\alpha_{j+1}\right)\,,
% \end{equation}
% as before, we consider periodic boundary conditions, and we follow the same notation.
% According to the Poisson brackets \eqref{eq:poisson_bracket} the equations of motion read

% \begin{equation}
% \label{eq:schur}
%     \dot{\alpha_{j}} = \rho_j^2(\alpha_{j+1}-\alpha_{j-1})\,, \rho_j = \sqrt{1-|\alpha_j|^2}
% \end{equation}
% which is the so-called Schur flow \cite{Golinskii}. In \cite{Ablowitz1975}, it is argued that the continuum limit of \eqref{eq:schur} is the modified Korteweg-de Vries equation:

% \begin{equation}
%     \partial_t u = \partial_x^3 u -6u^2\partial_x u\,.
% \end{equation}

{It is well known that the Schur flow admits as Lax matrix the same one as the AL \cite{Golinskii}, namely $\cE$ \eqref{eq:Lax_matrix} is the Lax matrix of the Schur flow.} This implies that the Ablowitz-Ladik's constants of motion are conserved also along the Schur flow \eqref{eq:schur}.

% Furthermore, if one chooses an initial data such that $\alpha_j(0)\in\R$ for all $j=1,\ldots,N$, then $\alpha_j(t)\in\R$ for all times. This, and the conservation of $K_0 = \prod_{j=1}^N \left( 1-|\alpha_j|^2\right)$, implies that we can choose as phase space for the Schur flow the $N$-cube $\mathbb{I}^N$, where $\I := (-1,1)$. 

Following the same construction {made for the Ablowitz-Ladik lattice}, on $\I^N$ we define the finite volume limit GGE as

	\begin{equation}
	\label{GGE Schur}
	   \di\mathbb{P}^{V,\beta}_{S,N}(\alpha_1,\ldots,\alpha_N)= \frac{1}{Z_N^{S}(V,\beta)}\prod_{j=1}^N(1-\alpha_j^2)^{\beta-1}\mathbf{1}_{\{\alpha_j\in\I\}}\exp(-\Tr(V(\mathcal{E})))\di\balpha, 
	\end{equation}
	where $Z_N^{S}(V,\beta)$ is the partition function of the system 
	$$Z_N^{S}(V,\beta) = \int_{\I^N}\prod_{j=1}^N(1-\alpha_j^2)^{\beta-1}\exp(-\Tr(V(\mathcal{E})))\di\balpha. $$

Since according to the measure \eqref{GGE Schur} the matrix $\cE$ is real, its eigenvalues come in pairs \cite{Simon2005}, meaning that  if $e^{i\theta_j}$ is an eigenvalue, then its conjugate $e^{-i\theta_j}$ is also an eigenvalue. This implies that for a system of size $N$ even, there are just $n = N/2$ independent eigenvalues. Following the same idea as in \cite{Killip2004}, it is more convenient to restrict the argument of the eigenvalues in $[0,\pi)$ and then consider $x_j = \cos(\theta_j)$, $j=1,\ldots,n$. In these variables, the empirical spectral measure $ \mu_n(\cE)$ reads:

\begin{equation}
\label{eq:empirical schur}
     \mu_n(\cE) = \frac{1}{n} \sum_{j=1}^n \delta_{x_j}\,, \quad x_j \in \I \,.
\end{equation}

As a corollary of Lemma \ref{Lemma weak ldp} and Proposition \ref{prop large deviations}, we obtain the existence of a large deviations principle for the sequence $( \mu_n(\cE))$, namely:

\begin{corollary}
Let $V\;:\; \I \to \R$ be continuous. Under $\mathbb{P}^{V,\beta}_{S,n}$ the sequence $( \mu_n(\cE))$ fulfils a large deviations principle with good, convex rate function $S_\beta^V(\mu) = h_\beta^V(\mu) - \inf_{\nu\in \cP(\I)}h_\beta^V(\nu)$, where

\begin{equation}
    h_\beta^V(\nu) = K_\beta(\nu) + \int_\I V \di \nu\,,
\end{equation}
where $K_\beta(\nu)$ is the rate function of $\mu_n$ under the law $\mathbb{P}^{0,\beta}_{S,n}$\,.
\end{corollary}

\subsection{Jacobi beta ensemble in the high temperature regime}

The Jacobi beta ensemble refers to the distribution of charges constrained to the segment $\I$, and subjected to an external potential $W(x) = -a\ln(1-x) -b\ln(1+x) + V(x)$, here $a,b>-1$ and $W(x)\in C^0(\I)$. Specifically the joint distribution of these particles is

\begin{equation}
    \label{eq:JbE}
    \di \mathbb{P}_{J,n}^{(V,\wt \beta)} = \frac{1}{Z_{N}^J(V,\wt \beta)} \prod_{i<j}|x_i-x_j|^{\wt \beta}\prod_{j=1}^n(1-x_j)^a(1+x_j)^be^{-V(x_j)}\di x_j\,.
\end{equation}

In \cite{Killip2004}, Killip and Nenciu were able to show that the distribution \eqref{eq:JbE} can be realized as the eigenvalues distribution of a particular CMV matrix, specifically they proved the following

\begin{theorem}[cf. \cite{Killip2004} Proposition 5.3]

Let $N=2n$, consider the CMV matrix $E$ in \eqref{E} with parameters $\alpha_1,\ldots,\alpha_{2n-1}\in \I$ distributed according to

\begin{equation}
   \di \mathbb{B}_n^{(V,\wt \beta)} =  \frac{1}{\fZ_n(V,\wt \beta)}\prod_{j=1}^{2n-1}(1-\alpha_j^2)^{\wt \beta(2n-j)/4 -1}\prod_{j=1}^{2n-1}(1-\alpha_j)^{a+1-\wt\beta/4}(1+(-1)^{j+1}\alpha_j)^{b+1-\wt\beta/4}e^{\Tr{V(E)}}\di\alpha_j\,,
\end{equation}
and $\alpha_{2n} = -1$, here $\fZ_N(V,\beta)$ is the normalization constant. Then all the eigenvalues of $E$ come in pairs, meaning that if $e^{i\theta_j}$ is an eigenvalue, then also $e^{-i\theta_j}$ is one. Moreover, under the change of variables $\cos(\theta_j) = x_j$, the $x_j$s are distributed according to \eqref{eq:JbE}.
\end{theorem}

\begin{remark}
    We notice that the previous proposition is not stated in this way in \cite{Killip2004}, but this equivalent formulation is more useful for our purpose.
\end{remark}

Also in this case, we are interested in the high temperature regime for this ensemble. Specifically we consider the situation $\wt \beta = \frac{4\beta}{N} = \frac{2\beta}{n}$, and $a =b = -1 + \frac{\wt \beta}{4}$, in this regime $\di \mathbb{P}_{J,n}^{\left(V,\frac{\beta}{n}\right)}$ reads

\begin{equation}
    \label{eq:JbE_ht}
    \di \mathbb{P}_{J,n}^{\left(V,\frac{2\beta}{n}\right)} = \frac{1}{Z_{N}^J\left(V,\frac{ \beta}{n}\right)} \prod_{i<j}|x_i-x_j|^{\frac{2\beta}{n}}\prod_{j=1}^n(1-x_j)^{-1 +\frac{\beta}{2n}}(1+x_j)^{-1 +\frac{\beta}{2n}}e^{-V(x_j)}\di x_j\,, 
\end{equation}

and $\di \mathbb{B}_n^{\left(V,\frac{\beta}{n}\right)} $ becomes

\begin{equation}
  \di \mathbb{B}_n^{\left(V,\frac{\beta}{n}\right)}=  \frac{1}{\fZ_n\left( V,\frac{\beta}{n}\right)}\prod_{j=1}^{2n-1}(1-\alpha_j^2)^{\beta\left(1-\frac{j}{2n}\right) -1}\prod_{j=1}^{2n-1}e^{\Tr{V(E)}}\di\alpha_j\,.
\end{equation}

We mention that this particular regime was considered in \cite{Forrester2021,Trinh2021}. In these papers the authors computed the density of states for this ensemble in the case $V = 0$. 

We can apply \cite[Corollary 1.3]{Garcia} to \eqref{eq:JbE_ht} to obtain a large deviations principle for the empirical measure $\mu_n(E) = \frac{1}{n}\sum_{j=1}^n\delta_{x_j}$. Specifically, we deduce that

\begin{proposition}
For any continuous $V\; :\; \I \to \R$. The law of the empirical measures $\mu_n(E)$ under $\di \mathbb{P}_{J,n}^{\left(V,\frac{2\beta}{n}\right)}$ satisfies a large deviations principle at speed $n$ in the space $\cP(\I)$, with a good rate function {$\mu \mapsto Q_\beta^V(\mu)$ given for $\mu$ absolutely continuous with respect to Lebesgue measure, and with density $\frac{\di \mu}{\di x}$, by $ Q_\beta^V(\mu)= q_\beta^V(\mu) - \inf_{\nu\in\cP(\I)}q_\beta^V(\nu)$, where 

\begin{equation}
\label{eq: functional jacobi ht}
    q_\beta^V(\mu) = \int_\I (V(x) + \ln(1+x) + \ln(1-x)) \di \mu(x) -2\beta\int_{\I \times \I} \ln(|x-y|)\di \mu(x)\di \mu(y) + \int_\I \ln\left(\frac{\di \mu}{\di x}(x)\right)\di \mu(x)\,,
\end{equation}
and $ Q_\beta^V(\mu)= +\infty $ otherwise.}
\end{proposition}

We notice that the arguments in Section \ref{sec Circular} and \ref{sec main thm} can be applied also in this context with $\di \mathbb{P}_{J,n}^{\left(V,\frac{2\beta}{n}\right)}$ in place of $\di \mathbb{P}_{C,N}^{\left(V,\frac{2\beta}{N}\right)}$, and $\di\mathbb{P}^{V,\beta}_{S,N}$ in place of $\di\mathbb{P}^{V,\beta}_{AL,N}$. Hence, we deduce the following result

\begin{theorem}
Consider the sequence of measures $\mu_n(\cE)$ \eqref{eq:empirical schur} under the law $\di \mathbb{P}_{S,2n}^{V,\beta}$ \eqref{GGE Schur}, then
\begin{equation}
     \mu_n(\cE) \xrightarrow{\textrm{a.s.}} \nu_\beta^V\,.
\end{equation}
Moreover, $\nu_\beta^V$ is absolutely continuous with respect to the Lebesgue measure, and it reads

\begin{equation}
    \nu_\beta^V = \partial_\beta(\beta \mu_\beta^V)\,,
\end{equation}
where $\mu_\beta^V$ is the unique minimizer of the functional $q_\beta^V$ \eqref{eq: functional jacobi ht}.
\end{theorem}

Finally, it is worth to mention that in the case $V(x) = 0$, it is possible to compute explicitly the densities of states for both the Jacobi beta ensemble at high temperature and for the Schur flow \cite{mazzuca2021mean,Forrester2021}.
}

\appendix

{
\section{Technical Results}
\label{app:A}
In this appendix we collect the proof of all the technical results that we exploit along the proof of the main theorem. For reader convenience, we report here the statement of Lemmas.

\subsection*{Proof of Lemma \ref{LEM:DISTANCE_INEQ}}

\begin{Lemma}
For any $A$, $B$ unitary matrices of size $N\times N$, we have
\begin{itemize}
    \item For $f$ with bounded variation, 
    $$\left| \int fd\mu - \int fd\nu \right| \leq \|f\|_{\text{BV}}\frac{rank(A-B)}{N},$$
    \item For $f$ Lipschitz,
    $$\left| \int fd\mu - \int fd\nu \right| \leq \|f\|_{\text{Lip}}\frac{1}{N}\sum_{i,j}|(A-B)_{i,j}|.$$
\end{itemize}
As a consequence,
\begin{equation}
    \label{distance}
    d(\mu(A),\mu(B))\leq \min\left\{\frac{rank(A-B)}{N},\frac{1}{N}\sum_{i,j}|(A-B)_{i,j}| \right\}.
\end{equation}
\end{Lemma}

\begin{proof}
The first point is a consequence of the fact that the eigenvalues of $A$ and $B$ interlace on the unit circle.

First, we order the eigenvalues $\lambda_1(A), \ldots, \lambda_N(A), \lambda_1(B), \ldots, \lambda_N(B)$ of $A,B$ in such a way that

\begin{equation}
-\pi \leq \arg(\lambda_1(A)) \leq \ldots \leq \arg(\lambda_N(A)) < \pi\,,
\end{equation}
and analogously for $B$.

Write $B=(I_N+(B-A)A^{-1})A$ and set $U:=I_N+(B-A)A^{-1}$. One checks that $U$ is unitary, $B=UA$, and that $\text{rank}(U-I)= \text{rank}(B-A)=:r$. By \cite[section 6, equation (85)]{ArbenzGolub}, we deduce that for $1\leq j \leq N$
%\begin{equation}
%\label{interlacing}
%        \arg(\lambda_{j+r}(A)) \leq \arg(\lambda_j(B)) \leq \arg(\lambda_{j-r}(A)), \quad j \leq \min\left\{ j-1, N-j \right\}\,.
%\end{equation}
\begin{equation}
\label{interlacing}
        \arg(\lambda_{j-r}(A)) \leq \arg(\lambda_j(B)) \leq \arg(\lambda_{j+r}(A))\,.
\end{equation}

This means that $\lambda_j(B)$ lies on the anticlockwise arc $(\arg(\lambda_{j-r}(A)),\arg(\lambda_{j+r}(A)))$ of the circle. If $j-r\leq 0$ we identify $\lambda_{j-r}$ with $\lambda_{j-r+N}$, and analogously for the case $j+r>N$.\\
It is a classical result (see \cite{Alicebook}) to deduce from \eqref{interlacing} that
$$ \left|\int fd\mu_N(A) - \int fd\mu_N(B) \right| \leq \|f\|_{\text{BV}}\frac{r}{N} = \frac{r}{N}\,, $$
for any $f:\T \to \R$ such that $||f||_{BV}\leq 1$. As a consequence, we obtain the first point.\\
The proof of the second point is the same as in the symmetric case, see \cite[(16)]{GMToda}. Indeed, we only use the fact that a normal matrix is unitarily diagonalizable.
\end{proof}

\subsection*{Proof of Lemma \ref{lem:estimate_product}}

\begin{Lemma}
Let $N=2k$ be even and $A$ be a $N\times N$ matrix. Then, \begin{itemize}
    \item $\sum_{i,j}|(\cL A)_{i,j}| \leq 2\sum_{i,j}|A_{i,j}|$,
    \item $\sum_{i,j}|(A \cM)_{i,j}| \leq 2\sum_{i,j}|A_{i,j}|,$
\end{itemize}
{where $\cM$, and $\cL$ are defined in \eqref{eq:cM_cL}.}
\end{Lemma}

\begin{proof}
We will just prove the first point, since the proof of the second one follows the same lines.

For $0\leq l \leq k-1$ and $1\leq j\leq N$, consider
$$ (\cL A)_{2l+1,j} = \overline{\alpha}_{2l+1}A_{2l+1,j}+\rho_{2l+1}A_{2l+2,j} \quad \text{and}\quad (\cL A)_{2l+2,j} = \rho_{2l+1}A_{2l+1,j}-\alpha_{2l+2}A_{2l+2,j}.$$
Summing over $i,j$,
$$\sum_{i,j}\left|(\cL A)_{i,j}\right|=\sum_{l=0}^{k-1}\sum_{j=1}^N \left|(\cL A)_{2l+1,j}\right|+\left|(\cL A)_{2l+2,j}\right|\leq 2\sum_{l=0}^{k-1}\sum_{j=1}^N|A_{2l+1,j}|+|A_{2l+2,j}|=2\sum_{i,j}|A_{i,j}|, $$
where we used that $|\alpha_i|,\rho_i \leq 1$.
\end{proof}

\subsection*{Proof of Lemma \ref{lem:bomb}}
{
\begin{Lemma}
Let $\alpha_\nu $ and $\alpha_{\nu+h}$ defined by equation \eqref{eq_coupling}. Define $\rho_\nu = \sqrt{1-|\alpha_\nu|^2}$, and $\rho_{\nu+h} = \sqrt{1-|\alpha_{\nu+h}|^2}$, then the following hold

\begin{enumerate}[i.]
    \item \begin{equation}
        \begin{split}
            &|\alpha_\nu - \alpha_{\nu+h}| \leq \frac{Y_h}{(X_1^2 + X_2^2+Y_h^2)^\frac{1}{2}}\,,\, \text{almost surely}\,,\\
            &|\rho_\nu - \rho_{\nu+h}| \leq \frac{Y_h}{(X_1^2 + X_2^2+Y_h^2)^\frac{1}{2}}\,,\, \text{almost surely}\,,
        \end{split}
    \end{equation}
    where $X_1,X_2\sim \cN(0,1)$, $Y_h\sim \chi_h$ are all independent.
    
    \item define $Z_h = \frac{Y_h}{(X_1^2 + X_2^2+Y_h^2)^\frac{1}{2}}$, and $a(h) = -\frac{1}{2}\ln(h) +1$, then there exists a constant $K$ independent of $h$ such that
    \begin{equation}
    \label{eq:sup_bound}
    \sup_{0<h<1}\meanval{\exp(a(h) Z_h)} \leq K\,.
\end{equation}
\end{enumerate}
\end{Lemma}
}

\begin{comment}
\begin{Lemma}
Let $\alpha_\nu \sim \Theta_\nu$ \eqref{def theta}, and $\alpha_{\nu+h} \sim \Theta_{\nu +h}$. Define $\rho_\nu = \sqrt{1-|\alpha_\nu|^2}$, and $\rho_{\nu+h} = \sqrt{1-|\alpha_{\nu+h}|^2}$, then the following hold

\begin{enumerate}[i.]
    \item \begin{equation}
        \begin{split}
            &|\alpha_\nu - \alpha_{\nu+h}| \leq \frac{Y_h}{(X_1^2 + X_2^2+Y_h^2)^\frac{1}{2}}\,,\\
            &|\rho_\nu - \rho_{\nu+h}| \leq \frac{Y_h}{(X_1^2 + X_2^2+Y_h^2)^\frac{1}{2}}\,,
        \end{split}
    \end{equation}
    where $X_1,X_2\sim \cN(0,1)$, $Y_h\sim \chi_h$ are all independent.
    
    \item define $Z_h = \frac{Y_h}{(X_1^2 + X_2^2+Y_h^2)^\frac{1}{2}}$, and $a(h) = -\frac{1}{4}\ln(h) +1$, then there exists a constant $K$ independent of $h$ such that
    \begin{equation}
    \label{eq:sup_bound}
    \sup_{0<h<1}\meanval{\exp(a(h) Z_h)} \leq K\,.
\end{equation}
\end{enumerate}
\end{Lemma}
\end{comment}

\begin{proof}
First, we focus on claim $i.$.
\begin{comment}
{ We recall that in view of Lemma \ref{lem_representation},  we have the following representation for the distribuition of $\alpha_\nu,\alpha_{\nu+h}$ }
\end{comment}
We recall that $\alpha_\nu,\alpha_{\nu+h}$ are defined by

\begin{equation}
    \alpha_\nu {=} \frac{X_1 + iX_2}{(X_1^2 + X_2^2 + Y_{\nu-1}^2)}\,, \quad \alpha_{\nu+h} {=} \frac{X_1 + iX_2}{(X_1^2 + X_2^2 + Y_{\nu-1}^2 + Y_h^2)}\,.
\end{equation}
From the previous equation, we deduce that
\begin{align*}
|\alpha_\nu-\alpha_{\nu+h}| &= \frac{|X_1+iX_2|}{(X_1^2+X_2^2+Y_{\nu-1}^2)^\frac{1}{2}}\left(1-\left(\frac{X_1^2+X_2^2+Y_{\nu-1}^2}{X_1^2+X_2^2+Y_{\nu-1}^2+Y_h^2}\right)^\frac{1}{2}\right)\\
   &=\frac{|X_1+iX_2|}{(X_1^2+X_2^2+Y_{\nu-1}^2)^\frac{1}{2}}\left(1-\left(1-\frac{Y_h^2}{X_1^2+X_2^2+Y_{\nu-1}^2+Y_h^2}\right)^\frac{1}{2}\right)\\
   &\leq \left(\frac{X_1^2+X_2^2}{X_1^2+X_2^2+Y_{\nu-1}^2}\right)^\frac{1}{2} \frac{Y_h}{(X_1^2+X_2^2+Y_{\nu-1}^2+Y_h^2)^\frac{1}{2}},
\end{align*}
where we used in the previous line that for $0\leq a\leq b$ we have 
\begin{equation}
\label{eq_sqroot}
    \sqrt{b}\leq \sqrt{b-a}+\sqrt{a}
\end{equation}
and we took $a=\frac{Y_h^2}{X_1^2+X_2^2+Y_{\nu-1}^2+Y_h^2}$, $b=1$. The last term is bounded by the announced bound.

One can proceed analogously for $|\rho_\nu - \rho_{\nu + h}|$ obtaining that

\begin{equation}
    \begin{split}
        |\rho(\alpha_{\nu+h})-\rho(\alpha_\nu)|& =\sqrt{1-|\alpha_{\nu+h}|^2}-\sqrt{1-|\alpha_\nu|^2} \leq \sqrt{|\alpha_\nu|^2-|\alpha_{\nu+h}|^2} \\
        & = \sqrt{\frac{X_1^2 + X_2^2}{X_1^2 + X_2^2 + Y_{\nu-1}^2}} \sqrt{1 - \frac{X_1^2+X_2^2 + Y_{\nu-1}^2}{X_1^2+X_2^2 + Y_{\nu-1}^2 + Y_h^2}} \leq \frac{Y_h}{(X_1^2+X_2^2+Y_{\nu-1}^2+Y_h^2)^\frac{1}{2}}
    \end{split}
\end{equation}

where we used again equation \eqref{eq_sqroot} with $a=1-|\alpha_\nu|^2$ and $b=1-|\alpha_{\nu+h}|^2$. Thus, point $i.$ is proved.

To prove point $ii.$, we find explicitly the law of $Z_h$. Thus, we consider a continuous function $f : \;(0,1) \; \to \R$, and we compute:

\begin{equation}
    \int_{\R^2 \times \R_+} f\left( \frac{y}{\left(x_1 ^2 + x_2^2 + y^2 \right)^\frac{1}{2}}\right)e^{-\frac{x_1^2 + x_2^2+y^2}{2}}y^{h-1}\di x_1\di x_2 \di y\,.
\end{equation}
Performing the change of coordinates $(u,v)=\frac{1}{(x_1^2+x_2^2+y^2)^{1/2}}(x_1,x_2)$, which is the same one that we performed in Lemma \ref{lem_representation}, we obtain that

\begin{equation}
    \begin{split}
            &\int_{\R^2 \times \R_+} f\left( \frac{y}{\left(x_1 ^2 + x_2^2 + y^2 \right)^\frac{1}{2}}\right)e^{-\frac{x_1^2 + x_2^2+y^2}{2}}y^{h-1}\di x_1\di x_2 \di y
            \\ &=\int_{\D\times\R_+ }\frac{f\left( \sqrt{1 - u^2-v^2}\right)}{(1-u^2-v^2)^2} e^{-\frac{y^2}{2(1-u^2-v^2)}}y^{h+1}\di u \di v \di y\\
            & \stackrel{\sqrt{1-u^2-v^2}t = y }{=} \int_{\D\times {\R_+}}f\left( \sqrt{1-u^2-v^2}\right)\left(1-u^2-v^2\right)^{\frac{h}{2} - 1}e^{-\frac{t^2}{2}}t^{h+1}\di u \di v \di t\,.
    \end{split}
\end{equation}
We can now explicitly compute the integral in $t$. Moreover, we can express the remaining part of the integral in polar coordinates; namely, we apply the change of variables $u = \rho \cos(\theta), v = \rho \sin(\theta)$, obtaining that:

\begin{equation}
    \begin{split}
        & \int_{\D\times {\R_+}}f\left( \sqrt{1-u^2-v^2}\right)\left(1-u^2-v^2\right)^{\frac{h}{2} - 1}e^{-\frac{t^2}{2}}t^{h+1}\di u \di v \di t \\ &= 2\pi 2^{\frac{h}{2}}\Gamma\left( \frac{h}{2} + 1\right)\int_0^1 \rho f(\sqrt{1-\rho^2})\left( 1-\rho^2\right)^{\frac{h}{2} -1 }\di \rho\\
        & \stackrel{\sqrt{1-\rho^2} = {w} }{=} 2\pi 2^{\frac{h}{2}}\Gamma\left( \frac{h}{2} + 1\right) \int_0^1 f(w)w^{h-1}{\di w}\,,
    \end{split}
\end{equation}
here $\Gamma(x)$ is the gamma function \eqref{eq gamma}. Thus, in order to obtain the estimate \eqref{eq:sup_bound}, we have to deduce an upper bound for

\begin{equation}
    \sup_{0<h<1}\frac{\int_0^1 e^{a(h)w} w^{h-1}\di w}{\int_0^1w^{h-1}\di w}\,.
\end{equation}

For any $0<h<1$, we can explicitly compute the denominator as
\begin{equation}
\label{eq:denominator}
    \int_0^1w^{h-1}\di w = \frac{1}{h}\,.
\end{equation}
Moreover, we can give an upper bound on the numerator as
\begin{equation}
\label{eq:numerator}
    \begin{split}
        \int_0^1 e^{a(h)w} w^{h-1}\di w & \stackrel{a(h)w = r}{=}\frac{1}{a(h)^h}\int_0^{a(h)}e^r r^{h-1} \di r = \frac{1}{a(h)^h}\left( \int_0^1 e^r r^{h-1}\di r + \int_1^{a(h)}e^r r^{h-1}\di r\right)\\
        & \leq \frac{1}{a(h)^h}\left(e\int_0^1 r^{h-1}\di r + \int_1^{a(h)} e^r\di r  \right) \leq \frac{{e}}{a(h)^h h} + \frac{e^{a(h)}}{a(h)^h}\,.
    \end{split}
\end{equation}
Combining \eqref{eq:denominator}-\eqref{eq:numerator}, with our choice of $a(h) = -\frac{1}{2}\ln(h) + 1$, we deduce that there exists a constant $K$ independent of $h$ such that \eqref{eq:sup_bound} holds. \\

\end{proof}

\subsection*{Proof of Lemma \ref{Lemma_minimizers}} 

\begin{Lemma}
Let $M_\beta^V = (J_\beta^V)^{-1}(\{0\})$ be the set of minimizers of $J_\beta^V$. Then, $\beta \mapsto M_\beta^V$ is continuous in the sense that for all $\varepsilon>0$, there exists $\delta>0$ such that for all $0\leq h \leq \delta$, $M_{\beta +h}^V \subset (M_\beta^V)^\varepsilon$, where for $A\subset \mathcal{P}(\T)$ we denote $A^\varepsilon = \{\mu\in \mathcal{P}(\T)\ |\ d(\mu,A)\leq \varepsilon \}.$
\end{Lemma}

\begin{proof}
Let $\varepsilon>0$. We are going to show that for $h>0$ small enough, we have 
$$ -\inf_{\left[(M_\beta^V)^\varepsilon\right]^c} J^V_{\beta+h} < 0,$$
which will ensure that $J^V_{\beta+h}>0$ on $\left[(M_\beta^V)^\varepsilon\right]^c$, thus $\left[(M_\beta^V)^\varepsilon\right]^c \subset \left[(M_{\beta+h}^V)\right]^c,$ and hence the conclusion.\\
% As in Lemma \ref{cor: LDP AL}, we write 
% $$ \di \mathbb{P}^{V,\beta}_{AL,N}=\frac{Z_{N}^{AL}(0,\beta)}{Z_{N}^{AL}(V,\beta)}e^{-N\int_\T V\di\mu_N}\di\mathbb{P}^{\beta}_{AL,N}. $$
By the large deviations principle for $(\mu_N)_{N\text{ even}}$ under $\mathbb{P}^{V,\beta}_{AL,N}$, Corollary \ref{cor: LDP AL}, since $\left[(M_\beta^V)^\varepsilon\right]^c$ is open, we have
\begin{align*}
    -\inf_{\left[(M_\beta^V)^\varepsilon\right]^c}J^V_{\beta+h} &\leq \liminf_{N\text{ even}} \frac{1}{N}\ln \mathbb{P}^{V,\beta+h}_{AL,N}\left(\mu_N(\cE) \in \left[(M_\beta^V)^\varepsilon\right]^c \right) \\
                             &= \liminf_{N\text{ even}} \frac{1}{N}\ln \mathbb{P}^{V,\beta+h}_{AL,N}\left(d(\mu_N(\cE),M_\beta^V) > \varepsilon \right)\\& 
                             \leq \limsup_{N\text{ even}} \frac{1}{N}\ln \mathbb{P}^{V,\beta+h}_{AL,N}\left(d(\mu_N(\cE),M_\beta^V) \geq \varepsilon \right)\,.
\end{align*}
Since for any positive $h$ and $\balpha\in\D^N$ $\prod_{j=1}^N (1-|\alpha_j|^2)^h \leq 1$ ,we deduce that for any $A\subseteq\D^N$

\begin{equation}
    \frac{1}{N}\ln\left( \mathbb{P}^{V,\beta+h}_{AL,N}\left( A\right) \right) \leq \frac{1}{N}\left( \ln\left(\frac{Z_N^{AL}(V,\beta)}{Z_N^{AL}(V,\beta+h)}\right) + \ln\left(\mathbb{P}_{AL,N}^{V,\beta}(A)\right)\right)\,,
\end{equation}
we recall that $\mathbb{P}_{AL,N}^{V,\beta}$ is defined in \eqref{GGE AL}. 

Applying the previous inequality in the case $A = \{d(\mu_N(\cE),M_\beta^V) \geq \varepsilon\}$, we conclude that
\begin{equation}
    -\inf_{\left[(M_\beta^V)^\varepsilon\right]^c}J^V_{\beta+h} \leq \limsup_{N \to \infty} \frac{1}{N}\left( \ln\left(\frac{Z_N^{AL}(V,\beta)}{Z_N^{AL}(V,\beta+h)}\right) + \ln\left(\mathbb{P}_{AL,N}^{V,\beta}(d(\mu_N(\cE),M_\beta^V) \geq \varepsilon)\right)\right)\,.
\end{equation}

From Corollary \ref{cor: LDP AL}, we deduce that there exists a positive constant $c$, independent of $h$, such that

\begin{equation}
    \limsup_{N \to \infty}\frac{1}{N}\mathbb{P}_{AL,N}^{V,\beta}(d(\mu_N(\cE),M_\beta^V) \geq \varepsilon)) \leq -\inf_{\wo{\left[(M_\beta^V)^\varepsilon\right]^c}}J^V_{\beta} < -c\,.
\end{equation}
Thus, to conclude we have just to prove that the function $g(\beta)= \lim_{N\to\infty}\frac{1}{N}\ln\left( Z_N^{AL}(V,\beta)\right)$ is continuous in $\beta$. Actually, we prove that this function is convex in $\beta$. Let $1/p + 1/q = 1$, and $\beta_1,\beta_2\in\R_+$ then
\begin{equation}
\begin{split}
        Z_N^{AL}\left(V,\frac{\beta_1}{p} + \frac{\beta_2}{q}\right) & = \int_{\D^N}\prod_{j=1}^N(1-|\alpha_j|^2)^{\frac{\beta_1}{p} + \frac{\beta_2}{q} - 1}\exp(-\Tr(V(\mathcal{E})))\di^2\balpha \\
         &= \int_{\D^N}\prod_{j=1}^N(1-|\alpha_j|^2)^{\frac{\beta_1-1}{p} + \frac{\beta_2-1}{q}}\exp\left(-\left(\frac{1}{p} + \frac{1}{q}\right)\Tr(V(\cE))\right)\di^2\balpha\\
         & \leq Z_N^{AL}\left(V,\beta_1\right)^{\frac{1}{p}}Z_N^{AL}\left(V,\beta_2\right)^{\frac{1}{q}}\,,
\end{split}
\end{equation}
where in the last inequality we used H\"older inequality. This implies that

\begin{equation}
    g\left( \frac{\beta_1}{p} + \frac{\beta_2}{q}\right) \leq \frac{1}{p} g(\beta_1) + \frac{1}{q} g(\beta_2)\,,
\end{equation}
thus $g(\beta)$ is convex, and so continuous, for $\beta > 0$. We can now choose $h$ is such a way that
$$ \Bigg\vert\limsup_{N \to \infty} \frac{1}{N} \ln\left(\frac{Z_N^{AL}(V,\beta)}{Z_N^{AL}(V,\beta+h)}\right)\Bigg\vert < c\,,$$
so we obtain that
\begin{equation}
    \inf_{\left[(M_\beta^V)^\varepsilon\right]^c}J^V_{\beta+h} > 0\,.
\end{equation}

\end{proof}

\paragraph{Acknoledgments}\hfill
\newline
This material is based upon work supported by the National Science Foundation under Grant No. DMS-1928930 while the author participated in a program hosted by the Mathematical Sciences Research Institute in Berkeley, California, during the Fall 20-21 semester  	"Universality and Integrability in Random Matrix Theory and Interacting Particle Systems".
 G.M. has received funding from the European Union's H2020 research and innovation programme under the Marie Sk\l odowska--Curie grant No. 778010 {\em  IPaDEGAN}.}

\bibliographystyle{siam}

\bibliography{bibAbel3}

\end{document}